\documentclass[12pt]{amsart}

\usepackage{amsmath}
\usepackage{amssymb}
\usepackage{ascmac}
\usepackage{amsthm}

\makeatletter

\@addtoreset{equation}{section}
\makeatother

\theoremstyle{plain}
\newtheorem{thm}{Theorem}[section]
\newtheorem*{thm*}{Theorem}
\newtheorem{prop}[thm]{Proposition}
\newtheorem{lem}[thm]{Lemma}
\newtheorem{cor}[thm]{Corollary}

\theoremstyle{definition}

\theoremstyle{remark}
\newtheorem{rem}{{\rm Remark}}[section]
\newtheorem{ex}{\rm{Example}}[section]
\newtheorem{defi}[ex]{{\rm Definition}}

\newcommand{\vol}{\operatorname{vol}}
\newcommand{\ric}{\operatorname{Ric}}
\newcommand{\Div}{\operatorname{div}}
\newcommand{\bm}{\partial M}
\newcommand{\ball}{B^{n}_{\kappa,\lambda}}
\newcommand{\const}{C_{\kappa,\lambda}}
\newcommand{\bconst}{\bar{C}_{\kappa,\lambda}}
\newcommand{\bball}{\partial B^{n}_{\kappa,\lambda}}

\newcommand{\dm}{D(M,\partial M)}
\newcommand{\tr}{\operatorname{trace}}
\newcommand{\cut}{\mathrm{Cut}\,}

\newcommand{\expp}{\exp^{\perp}}
\newcommand{\inte}{\mathrm{Int}\,}
\newcommand{\tbp}{T^{\perp}\bm}

\newcommand{\thetae}{\theta}
\newcommand{\thetaex}{\theta}
\newcommand{\bthetae}{\bar{\theta}}
\usepackage{latexsym}
\def\qed{\hfill $\Box$} 

\title[rigidity of manifolds with boundary]{rigidity of manifolds with boundary under a lower Ricci curvature bound}
\author{Yohei Sakurai}
\date{\today}
\address{Graduate School of Pure and Applied Sciences, University of Tsukuba, Tennodai 1-1-1, Tsukuba, Ibaraki, 305-8577, Japan}
\email{sakurai@math.tsukuba.ac.jp}
\thanks{Research Fellow of Japan Society for the Promotion of Science for 2014-2016}
\begin{document}
\maketitle

\begin{abstract}
We study Riemannian manifolds with boundary under a lower Ricci curvature bound,
and a lower mean curvature bound for the boundary.
We prove a volume comparison theorem of Bishop-Gromov type concerning the volumes of the metric neighborhoods of the boundaries.
We conclude several rigidity theorems.
As one of them,
we obtain a volume growth rigidity theorem.
We also show a splitting theorem of Cheeger-Gromoll type under the assumption of the existence of a single ray.
\end{abstract}

\section{Introduction}\label{sec:introduction}
In this paper, 
we study Riemannian manifolds with boundary under a lower Ricci curvature bound,
and a lower mean curvature bound for the boundary.
Heintze and Karcher in \cite{HK},
and Kasue in \cite{K2} (\cite{K1}),
have proved several comparison theorems for such manifolds with boundary.
Furthermore,
Kasue has proved rigidity theorems in \cite{K3}, \cite{K4} for such manifolds with boundary (see also \cite{K5}, \cite{I}).
These rigidity theorems state that if such manifolds satisfy suitable rigid conditions,
then there exist diffeomorphisms preserving the Riemannian metrics between the manifolds and the model spaces.
Other rigidity results have been also studied in \cite{dCX} and \cite{X}, and so on.

In order to develop the geometry of such manifolds with boundary,
we prove a volume comparison theorem of Bishop-Gromov type concerning the metric neighborhoods of the boundaries,
and produce a volume growth rigidity theorem.
We also prove a splitting theorem of Cheeger-Gromoll type under the assumption of the existence of a single ray emanating from the boundary.
We obtain a lower bound for the smallest Dirichlet eigenvalues for the $p$-Laplacians.
We also add a rigidity result to the list of the rigidity results obtained by Kasue in \cite{K4} on the smallest Dirichlet eigenvalues for the Laplacians.

The preceding rigidity results mentioned above have stated the existence of Riemannian isometries between manifolds with boundary and the model spaces.
On the other hand,
our rigidity results discussed below states the existence of isometries as metric spaces from a view point of metric geometry.
These notions are equivalent to each other (see Subsection \ref{sec:metric}).
\subsection{Main results}
For $\kappa \in \mathbb{R}$,
we denote by $M^{n}_{\kappa}$ the $n$-dimensional space form with constant curvature $\kappa$,
and by $g^{n}_{\kappa}$ the standard Riemannian metric on $M^{n}_{\kappa}$.

We say that $\kappa\in \mathbb{R}$ and $\lambda \in \mathbb{R}$ satisfy the \textit{ball-condition}
if there exists a closed geodesic ball $\ball$ in $M^{n}_{\kappa}$ with non-empty boundary $\bball$
such that $\bball$ has a constant mean curvature $\lambda$.
We denote by $\const$ the radius of $\ball$.
We see that $\kappa$ and $\lambda$ satisfy the ball-condition if and only if either
(1) $\kappa>0$; 
(2) $\kappa=0$ and $\lambda>0$;
or (3) $\kappa<0$ and $\lambda>\sqrt{\vert \kappa \vert}$.
Let $s_{\kappa,\lambda}(t)$ be a unique solution of the so-called Jacobi-equation
\begin{equation*}
f''(t)+\kappa f(t)=0
\end{equation*}
with intial conditions $f(0)=1$ and $f'(0)=-\lambda$.
We see that $\kappa$ and $\lambda$ satisfy the ball-condition if and only if
the equation $s_{\kappa,\lambda}(t)=0$ has a positive solution;
in particular,
$\const=\inf \{ t>0  \mid s_{\kappa,\lambda}(t)=0 \}$.

We denote by $\mathbb{S}^{n-1}$ the $(n-1)$-dimensional standard unit sphere.
Let $ds_{n-1}^{2}$ be the canonical metric on $\mathbb{S}^{n-1}$.
For an arbitrary pair of $\kappa\in \mathbb{R}$ and $\lambda\in \mathbb{R}$,
we define an $n$-dimensional \textit{model space} $M^{n}_{\kappa,\lambda}$ \textit{with constant mean curvature boundary with Riemannian metric} $g^{n}_{\kappa,\lambda}$ as follows:
If $\kappa>0$, then we put $(M^{n}_{\kappa,\lambda},g^{n}_{\kappa,\lambda}):=(\ball,g^{n}_{\kappa}|_{\ball})$.
If $\kappa \leq 0$, then
\[
  (M^{n}_{\kappa,\lambda},g^{n}_{\kappa,\lambda}) := \begin{cases}
  (\ball,g^{n}_{\kappa}|_{\ball})                                                                 & \text{if $\lambda>\sqrt{\vert \kappa \vert}$}, \\
  (M^{n}_{\kappa}\setminus \inte B^{n}_{\kappa,-\lambda} ,g^{n}_{\kappa}|_{M^{n}_{\kappa}\setminus \inte B^{n}_{\kappa,-\lambda}})     & \text{if $\lambda<-\sqrt{\vert \kappa \vert}$}, \\
  ([0,\infty) \times \mathbb{S}^{n-1},dt^{2}+s_{\kappa,\lambda}^{2}(t)ds_{n-1}^{2})            & \text{if $\vert \lambda \vert=\sqrt{\vert \kappa \vert}$},\\
  ([t_{\kappa,\lambda},\infty) \times \mathbb{S}^{n-1},dt^{2}+s_{\kappa,0}^{2}(t)ds_{n-1}^{2}) & \text{if $\vert \lambda \vert<\sqrt{\vert \kappa \vert}$},
  \end{cases}
\]
where $t_{\kappa,\lambda}$ is the unique solution of the equation $s_{\kappa,0}'(t)/s_{\kappa,0}(t)=-\lambda$
under the assumptions $\kappa<0$ and $\vert \lambda \vert<\sqrt{\vert \kappa \vert}$.
We denote by $h^{n-1}_{\kappa,\lambda}$ the induced Riemannian metric on $\partial M^{n}_{\kappa,\lambda}$.

For $n\geq 2$,
let $M$ be an $n$-dimensional, 
connected Riemannian manifold with boundary with Riemannian metric $g$.
The boundary $\bm$ is assumed to be smooth.
We denote by $h$ the induced Riemannian metric on $\bm$.
We say that
$M$ is \textit{complete}
if for the Riemannian distance $d_{M}$ on $M$
induced from the length structure determined by $g$,
the metric space $(M,d_{M})$ is complete.
We denote by $\ric_{g}$ the Ricci curvature on $M$ defined by $g$.
For $K \in \mathbb{R}$, 
by $\ric_{M}\geq K$, 
we mean that the infimum of $\ric_{g}$ on the unit tangent bundle on the interior $\inte M$ of $M$ is at least $K$.
For $x\in \bm$, 
we denote by $H_{x}$ the mean curvature on $\bm$ at $x$ in $M$.
For $\lambda \in \mathbb{R}$, 
by $H_{\bm}\geq \lambda$, 
we mean $\inf_{x\in \bm} H_{x}\geq \lambda$.
Let $\rho_{\bm}:M\to \mathbb{R}$ be the distance function from $\bm$ defined as
\begin{equation*}
\rho_{\bm}(p):=d_{M}(p,\bm).
\end{equation*}
The \textit{inscribed radius of} $M$ is defined as
\begin{equation*}
\dm:=\sup _{p\in M}\rho_{\bm}(p).
\end{equation*}
For $r>0$, 
we put $B_{r}(\bm):=\{\,p\in M \mid \rho_{\bm}(p) \leq r\,\}$.
We denote by $\vol_{g}$ the Riemannian volume on $M$ induced from $g$.

One of the main results is the following volume comparison theorem:
\begin{thm}\label{thm:volume comparison}
For $\kappa \in \mathbb{R}$ and $\lambda\in \mathbb{R}$,
and for $n\geq 2$,
let $M$ be an $n$-dimensional, connected complete Riemannian manifold with boundary with Riemannian metric $g$
such that $\ric_{M}\geq (n-1)\kappa$ and $H_{\bm} \geq \lambda$.
Suppose $\bm$ is compact.
Then for all $r,R\in (0,\infty)$ with $r\leq R$, we have
\begin{equation*}\label{eq:volume comparison}
\frac{\vol_{g} B_{R}(\bm)}{\vol_{g} B_{r}(\bm)}\leq \frac{\vol_{g^{n}_{\kappa,\lambda}} B_{R}(\partial M^{n}_{\kappa,\lambda})}{\vol_{g^{n}_{\kappa,\lambda}} B_{r}(\partial M^{n}_{\kappa,\lambda})}.
\end{equation*}
\end{thm}
Theorem \ref{thm:volume comparison} is an analogue to the Bishop-Gromov volume comparison theorem (\cite{Gr1}, \cite{Gr2}).
What happens in the equality case can be described
by using the Jacobi fields along the geodesics
perpendicular to the boundary (see Remark \ref{rem:volume eq} and Proposition \ref{prop:half vol2}).
\begin{rem}
Theorem \ref{thm:volume comparison} is a relative volume comparison theorem.
Under the same setting as in Theorem \ref{thm:volume comparison},
Heintze and Karcher have proved in Theorem 2.1 in \cite{HK} that
the absolute volume comparison inequality
\begin{equation*}\label{eq:absolute volume comparison}
\frac{\vol_{g} B_{r}(\bm)}{\vol_{h} \bm}\leq \frac{\vol_{g^{n}_{\kappa,\lambda}} B_{r}(\partial M^{n}_{\kappa,\lambda})}{\vol_{h^{n-1}_{\kappa,\lambda}} \partial M^{n}_{\kappa,\lambda}}
\end{equation*}
holds for every $r>0$.
This inequality can be derived from Theorem \ref{thm:volume comparison}.
Similar volume comparison inequalities for submanifolds have been studied in \cite{HK}.
\end{rem}
\begin{rem}
Kasue has shown in Theorem A in \cite{K3} that
if $\kappa$ and $\lambda$ satisfy the ball-condition,
then $\dm\leq \const$ (see Lemma \ref{lem:shinitai2});
moreover,
if there exists a point $p_{0}\in M$ such that $\rho_{\bm}(p_{0})=\const$,
then $M$ is isometric to $\ball$ (see Theorem \ref{thm:ball rigid}).
\end{rem}
\begin{rem}
It has been recently shown in \cite{L} that
if $M$ is an $n$-dimensional, 
connected complete Riemannian manifold with boundary such that $\ric_{M}\geq 0$ and $H_{\bm}\geq \lambda>0$,
then $D(M,\bm)\leq C_{0,\lambda}$;
moreover,
if $\bm$ is compact,
then $M$ is compact,
and $D(M,\bm)=C_{0,\lambda}$ if and only if
$M$ is isometric to $B^{n}_{0,\lambda}$.
It has been recently proved in \cite{LW} that
for $\kappa<0$ and $\lambda>\sqrt{\vert \kappa \vert}$,
if $M$ is an $n$-dimensional,
connected complete Riemannian manifold with boundary such that $\ric_{M}\geq (n-1)\kappa$ and $H_{\bm}\geq \lambda$,
then $D(M,\bm)\leq \const$;
moreover,
if $\bm$ is compact,
then $D(M,\bm)= \const$ if and only if
$M$ is isometric to $\ball$.
A similar result has been proved in \cite{LW}
for manifolds with boundary under a lower Bakry-\'Emery Ricci curvature bound.
It has been also recently stated in \cite{G} that
if $\kappa\in \mathbb{R}$ and $\lambda\in \mathbb{R}$ satisfy the ball-condition,
and if $M$ is an $n$-dimensional,
connected complete Riemannian manifold with boundary such that $\ric_{M}\geq (n-1)\kappa$ and $H_{\bm}\geq \lambda$,
then $D(M,\bm)\leq \const$;
moreover,
if $\bm$ is compact,
then $M$ is compact,
and $D(M,\bm)= \const$ if and only if
$M$ is isometric to $\ball$.
\end{rem}
\begin{rem}
We prove Theorem \ref{thm:volume comparison}
by using a geometric study of the cut locus for the boundary,
and a comparison result for the Jacobi fields along geodesics perpendicular to the boundary.

For metric measure spaces,
Strum \cite{St}, and Ohta \cite{O1}, \cite{O2} have independently introduced the so-called measure contraction property
that is equivalent to a lower Ricci curvature bound for manifolds without boundary.
We prove a measure contraction inequality for manifolds with boundary (see Proposition \ref{prop:mcp1}).
Using our measure contraction inequality,
we give another proof of Theorem \ref{thm:volume comparison}.
\end{rem}
For $\kappa\in \mathbb{R}$ and $\lambda\in \mathbb{R}$,
if $\kappa$ and $\lambda$ satisfy the ball-condition,
then we put $\bar{C}_{\kappa,\lambda}:=C_{\kappa,\lambda}$;
otherwise,
$\bar{C}_{\kappa,\lambda}:=\infty$.
We define a function $\bar{s}_{\kappa,\lambda}:[0,\infty)\to \mathbb{R}$ by
\[
\bar{s}_{\kappa,\lambda}(t) := \begin{cases}
                                            s_{\kappa,\lambda}(t) & \text{if $t< \bar{C}_{\kappa,\lambda}$},\\
                                            0                     & \text{if $t\geq    \bar{C}_{\kappa,\lambda}$},
                                 \end{cases}
\]
and define a function $f_{n,\kappa,\lambda}:[0,\infty)\to \mathbb{R}$ by
\begin{equation*}
f_{n,\kappa,\lambda}(t):=\int^{t}_{0}\, \bar{s}^{n-1}_{\kappa,\lambda}(u)\,du.
\end{equation*}
For $\kappa\in \mathbb{R}$ and $\lambda\in \mathbb{R}$,
we define $[0,\bar{C}_{\kappa,\lambda})\times_{\kappa,\lambda} \partial M$ as the warped product $([0,\bar{C}_{\kappa,\lambda})\times \partial M, dt^{2}+s^{2}_{\kappa,\lambda}(t)h)$ with Riemannian metric $dt^{2}+s^{2}_{\kappa,\lambda}(t)h$,
and we put $d_{\kappa,\lambda}:=d_{[0,\bar{C}_{\kappa,\lambda})\times_{\kappa,\lambda} \partial M}$.

Theorem \ref{thm:volume comparison} yields the following volume growth rigidity theorem:
\begin{thm}\label{thm:volume growth distance rigidity}
For $\kappa \in \mathbb{R}$ and $\lambda\in \mathbb{R}$,
and for $n\geq 2$,
let $M$ be an $n$-dimensional Riemannian manifold with boundary with Riemannian metric $g$
such that $\ric_{M}\geq (n-1)\kappa$ and $H_{\bm} \geq \lambda$.
Suppose $\partial M$ is compact.
Let $h$ denote the induced Riemannian metric on $\bm$.
If
\begin{equation*}\label{eq:volume growth rigidity}
\liminf_{r\to \infty}\frac{\vol_{g}B_{r}(\bm)}{f_{n,\kappa,\lambda}(r)}\geq \vol_{h}\bm,
\end{equation*}
then the metric space $(M,d_{M})$ is isometric to $([0,\bconst)\times_{\kappa,\lambda}\bm,d_{\kappa,\lambda})$.
Moreover,
if $\kappa$ and $\lambda$ satisfy the ball-condition,
then $(M,d_{M})$ is isometric to $(\ball,d_{\ball})$.
\end{thm}
\begin{rem}
Under the same setting as in Theorem \ref{thm:volume growth distance rigidity},
by Theorem \ref{thm:volume comparison},
we always have the following (see Proposition \ref{prop:volume growth sup}):
\begin{equation*}
\limsup_{r\to \infty}\frac{\vol_{g}B_{r}(\bm)}{f_{n,\kappa,\lambda}(r)}\leq \vol_{h}\bm.
\end{equation*}
Theorem \ref{thm:volume growth distance rigidity} is certainly concerned with a rigidity phenomenon.
\end{rem}

\subsection{Splitting theorems}\label{sec:Splitting theorems}
Kasue in Theorem C in \cite{K3} has proved the following splitting theorem.
For $\kappa \leq 0$,
let $M$ be an $n$-dimensional,
connected complete Riemannian manifold with boundary
such that $\ric_{M}\geq (n-1)\kappa$ and $H_{\bm} \geq \sqrt{\vert \kappa \vert}$.
If $M$ is noncompact and $\bm$ is compact,
then $(M,d_{M})$ is isometric to $([0,\infty)\times_{\kappa,\sqrt{\vert \kappa \vert}}\bm,d_{\kappa,\sqrt{\vert \kappa \vert}})$.
The same result has been proved by Croke and Kleiner in Theorem 2 in \cite{CK}.

In \cite{K3},
the proof of the splitting theorem is based on the original proof of the Cheeger-Gromoll splitting theorem in \cite{CG}.
For a ray $\gamma$ on $M$,
let $b_{\gamma}$ be the busemann function on $M$ for $\gamma$.
The key points in \cite{K3} are to show the existence of a ray $\gamma$ on $M$ such that
for all $t\geq 0$ we have $\rho_{\bm}(\gamma(t))=t$,
and the subharmonicity of the function $b_{\gamma}-\rho_{\bm}$
in a distribution sense,
and to apply an analytic maximal principle (see \cite{GT}).
In \cite{CK},
the splitting theorem has been proved by using the Calabi maximal principle (\cite{C})
similarly to the elementary proof of the Cheeger-Gromoll splitting theorem developed by Eschenburg and Heintze in \cite{EH}.
It seems that the proof in \cite{CK} relies on the compactness of $\bm$.

Let $M$ be a connected complete Riemannian manifold with boundary.
For $x\in \bm$,
we denote by $u_{x}$ the unit inner normal vector at $x$.
Let $\gamma_{x}:[0,T)\to M$ be the geodesic with initial conditions $\gamma_{x}(0)=x$ and $\gamma_{x}'(0)=u_{x}$.
We define a function $\tau:\bm\to \mathbb{R}\cup \{\infty\}$ by
\begin{equation*}
\tau(x):=\sup \{\,t>0 \mid \rho_{\bm}(\gamma_{x}(t))=t\,\}.
\end{equation*}

We point out that
the following splitting theorem holds for the case where
the boundary is not necessarily compact.
\begin{thm}\label{thm:splitting}
For $n\geq 2$ and $\kappa \leq 0$,
let $M$ be an $n$-dimensional,
connected complete Riemannian manifold with boundary
such that $\ric_{M}\geq (n-1)\kappa$ and $H_{\bm} \geq \sqrt{\vert \kappa \vert}$.
Assume that for some $x\in \bm$,
we have $\tau(x)=\infty$.
Then $(M,d_{M})$ is isometric to $([0,\infty)\times_{\kappa,\sqrt{\vert \kappa \vert}}\bm,d_{\kappa,\sqrt{\vert \kappa \vert}})$.
\end{thm}
Theorem \ref{thm:splitting} can be proved by a similar way to that of the proof of the splitting theorem in \cite{K3}.
We give a proof of Theorem \ref{thm:splitting} in which we use the Calabi maximal principle.
Our proof can be regarded as an elementary proof of the splitting theorem in \cite{K3}.
\begin{rem}
In Theorem \ref{thm:splitting},
if $\bm$ is noncompact,
then we can not replace the assumption of $\tau$ with that of the existence of a single ray orthogonally emanating from the boundary.
For instance,
we put
\begin{equation*}
M:=\{(p,q)\in \mathbb{R}^{2}\mid p<0, p^{2}+q^{2}\leq 1\}\cup  \{(p,q)\in \mathbb{R}^{2}\mid p\geq 0, \vert q \vert\leq 1\}.
\end{equation*}
Observe that $M$ is a $2$-dimensional,
connected complete Riemannian manifold with boundary
such that $\ric_{M}=0$ and $H_{\bm} \geq 0$.
For all $x\in \bm$,
we have $\tau(x)=1$.
The geodesic $\gamma_{(-1,0)}$ is a ray in $M$.
On the other hand,
$M$ is not isometric to the standard product $[0,\infty)\times \bm$.
\end{rem}

\subsection{Eigenvalues}
Let $M$ be a Riemannian manifold with boundary with Riemannian metric $g$.
For $p\in [1,\infty)$,
the \textit{$(1,p)$-Sobolev space} $W^{1,p}_{0}(M)$ \textit{on} $M$ \textit{with compact support}
is defined as the completion of the set of all smooth functions on $M$ whose support is compact and contained in $\inte M$
with respect to the standard $(1,p)$-Sobolev norm.
Let $\Vert \cdot \Vert$ denote the standard norm induced from $g$,
and $\Div$ the divergence with respect to $g$.
For $p\in [1,\infty)$,
the \textit{$p$-Laplacian} $\Delta_{p}\, f$ for $f \in W^{1,p}_{0}(M)$ is defined as
\begin{equation*}
\Delta_{p}\,f:=-\Div \,\left(\Vert \nabla f \Vert^{p-2}\, \nabla f \right),
\end{equation*}
where the equality holds in a weak sense on $W^{1,p}_{0}(M)$.
A real number $\lambda$ is said to be a \textit{$p$-Dirichlet eigenvalue} for $\Delta_{p}$ on $M$
if we have a non-zero function $f$ in $W^{1,p}_{0}(M)$ such that $\Delta_{p} f=\lambda \vert f \vert^{p-2}\,f$ holds  on $\inte M$ in a weak sense on $W^{1,p}_{0}(M)$.
For $p\in [1,\infty)$,
the \textit{Rayleigh quotient} $R_{p}(f)$ \textit{for} $f \in W^{1,p}_{0}(M)$ is defined as
\begin{equation*}
R_{p}(f):=\frac{\int_{M}\, \Vert \nabla f \Vert^{p}\,d \vol_{g}}{\int_{M}\,  \vert f \vert^{p}\,d \vol_{g}}.
\end{equation*}
We put $\mu_{1,p}(M):=\inf_{f} R_{p}(f)$,
where the infimum is taken over all non-zero functions in $W^{1,p}_{0}(M)$.
The value $\mu_{1,2}(M)$ is equal to the infimum of the spectrum of $\Delta_{2}$ on $M$.
If $M$ is compact,
and if $p\in (1,\infty)$,
then $\mu_{1,p}(M)$ is equal to the infimum of the set of all $p$-Dirichlet eigenvalues for $\Delta_{p}$ on $M$.

Due to the volume estimate obtained by Kasue in \cite{K5},
we obtain the following:
\begin{thm}\label{thm:p-Laplacian1}
For $\kappa \in \mathbb{R},\,\lambda\in \mathbb{R}$ and $D\in(0,\bar{C}_{\kappa,\lambda}]$,
and for $n\geq 2$,
let $M$ be an $n$-dimensional,
connected complete Riemannian manifold with boundary
such that $\ric_{M}\geq (n-1)\kappa,\, H_{\bm} \geq \lambda$ and $\dm \leq D$.
Suppose $\partial M$ is compact.
Then for all $p\in (1,\infty)$,
we have
\begin{equation*}\label{eq:p-Laplacian1}
\mu_{1,p}(M)\geq (\,p\,C(n,\kappa,\lambda,D)\,)^{-p},
\end{equation*}
where $C(n,\kappa,\lambda,D)$ is a positive constant defined by
\begin{equation*}\label{eq:constant}
C(n,\kappa,\lambda,D):=\sup_{t\in [0,D)}\, \frac{\int^{D}_{t}\,  s^{n-1}_{\kappa,\lambda}(s)\, ds}{s^{n-1}_{\kappa,\lambda}(t)}.
\end{equation*}
\end{thm}
\begin{rem}\label{rem:boundary case}
In Theorem \ref{thm:p-Laplacian1},
since $\bm$ is compact,
$D(M,\bm)$ is finite
if and only if $M$ is compact (see Lemma \ref{lem:compact}).
We see that $C(n,\kappa,\lambda,\infty)$ is finite if and only if $\kappa<0$ and $\lambda=\sqrt{\vert \kappa \vert}$;
in this case,
we have
$C(n,\kappa,\lambda,D)=\left((n-1)\lambda \right)^{-1}\,\left(1-e^{-(n-1)\lambda\, D} \right)$;
in particular,
$(2\,C(n,\kappa,\lambda,\infty))^{-2}=\left((n-1)\lambda/2 \right)^{2}$.
\end{rem}
\begin{rem}
For compact manifolds with boundary of non-negative Ricci curvature,
similar lower bounds for $\mu_{1,p}$ to that in Theorem \ref{thm:p-Laplacian1} have been obtained in \cite{KN}, in \cite{Z1} and in \cite{Z2}.
\end{rem}
We recall the works of Kasue in \cite{K4} for compact manifolds with boundary.
Let $n\geq 2$, $\kappa,\lambda\in \mathbb{R}$
and $D\in (0,\bar{C}_{\kappa,\lambda}]\setminus \{\infty\}$.
Kasue has proved in Theorem 2.1 in \cite{K4} that
there exists a positive constant $\mu_{n,\kappa,\lambda,D}$ such that
for every $n$-dimensional,
connected compact Riemannian manifold $M$ with boundary such that $\ric_{M}\geq (n-1)\kappa, \,H_{\bm} \geq \lambda$ and $\dm\leq D$,
we have $\mu_{1,2}(M)\geq \mu_{n,\kappa,\lambda,D}$;
moreover,
in some extremal case,
the equality holds if and only if $M$ is isometric to some model space. 
The extremal case happens only if
$\kappa$ and $\lambda$ satisfy the ball-condition
or the condition that the equation $s'_{\kappa,\lambda}(t)=0$ has a positive solution.
Note that
the equation $s'_{\kappa,\lambda}(t)=0$ has a positive solution
if and only if either
(1) $\kappa=0$ and $\lambda=0$;
(2) $\kappa<0$ and $\lambda\in (0,\sqrt{\vert \kappa \vert})$;
or (3) $\kappa>0$ and $\lambda \in (-\infty,0)$.
Let
\begin{equation*}\label{eq:Kasue estimate}
\bar{\mu}_{n,\kappa,\lambda,D}:=\left(\, 4 \sup_{t\in (0,D)} \int^{D}_{t}\, s^{n-1}_{\kappa,\lambda}(s)\,ds\,\int^{t}_{0}\, s^{1-n}_{\kappa,\lambda}(s)\,ds  \, \right)^{-1}.
\end{equation*}
It has been shown in Lemma 1.3 in \cite{K4} that $\mu_{n,\kappa,\lambda,D} >\bar{\mu}_{n,\kappa,\lambda,D}$.
Therefore,
for every $n$-dimensional,
connected compact Riemannian manifold $M$ with boundary such that $\ric_{M}\geq (n-1)\kappa, \,H_{\bm} \geq \lambda$ and $\dm \leq D$,
we have $\mu_{1,2}(M)> \bar{\mu}_{n,\kappa,\lambda,D}$.
This estimate for $\mu_{1,2}$ is better than that in Theorem \ref{thm:p-Laplacian1}.

Let $n\geq 2$, $\kappa<0$ and $\lambda=\sqrt{\vert \kappa \vert}$.
The model space $M^{n}_{\kappa,\lambda}$ is non-compact.
For $t\in [0,\infty)$,
we put $\phi_{n,\kappa,\lambda}(t):=t\,e^{\frac{(n-1)\lambda t}{2}}$.
The smooth function $\phi_{n,\kappa,\lambda} \circ \rho_{\partial M^{n}_{\kappa,\lambda}}$ on $M^{n}_{\kappa,\lambda}$ satisfies
$R_{2}(\phi_{n,\kappa,\lambda} \circ \rho_{\partial M^{n}_{\kappa,\lambda}})=\left((n-1)\lambda/2 \right)^{2}$;
hence,
$\mu_{1,2}(M^{n}_{\kappa,\lambda}) \leq \left((n-1)\lambda/2 \right)^{2}$.
Notice that
the value $(2\,C(n,\kappa,\lambda,\infty))^{-2}$ in Theorem \ref{thm:p-Laplacian1} is equal to $\left((n-1)\lambda/2 \right)^{2}$ (see Remark \ref{rem:boundary case}).
Theorem \ref{thm:p-Laplacian1} implies $\mu_{1,2}(M^{n}_{\kappa,\lambda}) = \left((n-1)\lambda/2 \right)^{2}$.
Let $D\in (0,\infty)$.
As mentioned above,
we have already known in \cite{K4} that
for every $n$-dimensional,
connected compact Riemannian manifold $M$ with boundary such that $\ric_{M}\geq (n-1)\kappa, \,H_{\bm} \geq \lambda$ and $\dm \leq D$,
we have $\mu_{1,2}(M)> \bar{\mu}_{n,\kappa,\lambda,D}$.
The value $\bar{\mu}_{n,\kappa,\lambda,D}$ is equal to $\left((n-1)\lambda/2 \right)^{2}\,\left(1-e^{-(n-1)\lambda\, D/2} \right )^{-2}$,
and tends to $\mu_{1,2}(M^{n}_{\kappa,\lambda})$ as $D \to \infty$.

By using Theorem \ref{thm:p-Laplacian1} and the splitting theorem in \cite{K3},
we add the following result for not necessarily compact manifolds with boundary to the list of the rigidity results obtained in \cite{K4}.
\begin{thm}\label{thm:eigenvalue rigidity}
Let $\kappa<0$ and $\lambda:=\sqrt{\vert \kappa \vert}$.
For $n\geq 2$,
let $M$ be an $n$-dimensional,
connected complete Riemannian manifold with boundary
such that $\ric_{M}\geq (n-1)\kappa$ and $H_{\bm} \geq \lambda$.
Suppose $\partial M$ is compact.
Then for all $p\in (1,\infty)$,
we have
\begin{equation*}\label{eq:noncompact estimate}
\mu_{1,p}(M)\geq \left(\frac{(n-1)\lambda}{p}\right)^{p};
\end{equation*}
if the equality holds for some $p\in (1,\infty)$,
then $(M,d_{M})$ is isometric to $([0,\infty)\times_{\kappa,\lambda}\bm, d_{\kappa,\lambda})$;
moreover,
if $p=2$,
then the equality holds if and only if
$(M,d_{M})$ is isometric to $([0,\infty)\times_{\kappa,\lambda}\bm, d_{\kappa,\lambda})$.
\end{thm}
\begin{rem}
In Theorem \ref{thm:eigenvalue rigidity},
the author does not know whether
in the case of $p\neq 2$
the value $\mu_{1,p}([0,\infty)\times_{\kappa,\lambda}\bm)$ is equal to $((n-1)\lambda/p)^{p}$.
\end{rem}

Cheeger and Colding in Theorem 2.11 in \cite{CC} have proved the segment inequality for complete Riemannian manifolds under a lower Ricci curvature bound.
They have mentioned that
their segment inequality gives a lower bound for the smallest Dirichlet eigenvalue for the Laplacian on a closed ball.

Based on the proof of Theorem \ref{thm:volume comparison},
we prove a segment inequality of Cheeger-Colding type for manifolds with boundary (see Proposition \ref{prop:segment}).
Using our segment inequality,
we obtain a lower bound for $\mu_{1,p}$
smaller than the lower bound in Theorem \ref{thm:p-Laplacian1} (see Proposition \ref{prop:p-Laplacian2}).

\subsection{Organization}
In Section \ref{sec:Preliminaries},
we prepare some notations and recall the basic facts on Riemannian manifolds with boundary.

In Section \ref{sec:Cut locus for the boundary},
for a connected complete Riemannian manifold with boundary,
we study the basic properties of the cut locus for the boundary.
The basic properties seem to be well-known,
however,
they has not been summarized in any literature.
For the sake of the readers,
we discuss them in order to prove our results.

In Section \ref{sec:Comparison theorems},
by using the study of the cut locus for the boundary in Section \ref{sec:Cut locus for the boundary},
we prove Theorem \ref{thm:volume comparison}.

In Section \ref{sec:Volume growth rigidity},
we prove Theorem \ref{thm:volume growth distance rigidity}.
The rigidity follows from the study in the equality case in Theorem \ref{thm:volume comparison}.

In Section \ref{sec:splitting},
we prove Theorem \ref{thm:splitting}.

In Section \ref{sec:segment inequality},
we prove Theorems \ref{thm:p-Laplacian1} and \ref{thm:eigenvalue rigidity}.
We also prove a segment inequality (see Proposition \ref{prop:segment}).
After that,
we show the Poincar\'e inequality (see Lemma \ref{lem:Poincare}),
and we conclude Proposition \ref{prop:p-Laplacian2}.

In Section \ref{sec:Measure contraction property},
we prove a measure contraction inequality (see Proposition \ref{prop:mcp1}).
We also give another proof of Theorem \ref{thm:volume comparison}.

\subsection*{{\rm Addendum}}
After completing the first draft of this paper,
the author has been informed by Sormani of the paper \cite{P} written by Perales.
Let $M$ be a connected complete Riemannian manifold with boundary such that $\ric_{M}\geq 0$ and $H_{\bm}\geq \lambda$.
The paper \cite{P} contains a Laplacian comparison theorem for $\rho_{\bm}$ everywhere in a barrier sense,
a theorem of volume estimates of the metric neighborhoods of $\bm$,
and applications to studies of convergences of such manifolds with boundary.

\subsection*{{\rm Acknowledgements}}
The author would like to express his gratitude to Professor Koichi Nagano for his constant advice and suggestions.
The author would also like to thank Professor Takao Yamaguchi for his valuable advice.
The author would like to thank Professor Yong Wei for informing him of the paper \cite{LW},
Professor Christina Sormani for informing him of the paper \cite{P},
and Professor Takumi Yokota for informing him of the paper \cite{CK}.
The author would also like to thank Professor Atsushi Kasue for his valuable comments that lead some improvements of Theorems \ref{thm:p-Laplacian1} and \ref{thm:eigenvalue rigidity}.
The author is grateful to an anonymous referee of some journal for valuable comments.
One of the comments leads the author to the study of the measure contraction inequality.

\section{Preliminaries}\label{sec:Preliminaries}
We refer to \cite{BBI} for the basics of metric geometry,
and to \cite{S} for the basics of Riemannian manifolds with boundary.

\subsection{Metric spaces}
Let $(X,d_{X})$ be a metric space.
For $r>0$ and $A\subset X$,
we denote by $U_{r}(A)$ the open $r$-neighborhood of $A$ in $X$,
and by $B_{r}(A)$ the closed one.

For a metric space $(X,d_{X})$,
the length metric $\bar{d}_{X}$ is defined as follows:
For two points $x_{1},x_{2}\in X$,
we put $\bar{d}_{X}(x_{1},x_{2})$
to the infimum of the length of curves connecting $x_{1}$ and $x_{2}$ with respect to $d_{X}$.
A metric space $(X,d_{X})$ is said to be a \textit{length space} if $d_{X}=\bar{d}_{X}$.

Let $(X,d_{X})$ be a metric space.
For an interval $I\subset \mathbb{R}$,
let $\gamma:I\to X$ be a curve.
We say that $\gamma$ is a \textit{normal minimal geodesic}
if for all $s,t\in I$,
we have $d_{X}(\gamma(s),\gamma(t))=\vert s-t\vert$,
and $\gamma$ is a \textit{normal geodesic}
if for each $t\in I$,
there exists an interval $J\subset I$ with $t\in J$ such that $\gamma|_{J}$ is a normal minimal geodesic.
A metric space $(X,d_{X})$ is said to be a \textit{geodesic space} if
for every pair of two points in $X$,
there exists a normal minimal geodesic connecting them.
A metric space is \textit{proper} if all closed bounded subsets of the space are compact.
The Hopf-Rinow theorem for length spaces (see e.g., Theorem 2.5.23 in \cite{BBI}) states that 
if a length space $(X,d_{X})$ is complete and locally compact,
and if $d_{X}<\infty$,
then $(X,d_{X})$ is a proper geodesic space.

\subsection{Riemannian manifolds with boundary}
For $n\geq 2$, 
let $M$ be an $n$-dimensional,
connected Riemannian manifold with (smooth) boundary
with Riemannian metric $g$.
For $p\in \inte M$, 
let $T_{p}M$ be the tangent space at $p$ on $M$,
and let $U_{p}M$ be the unit tangent sphere at $p$ on $M$.
We denote by $\Vert \cdot \Vert$ the standard norm induced from $g$.
If $v_{1},\dots,v_{k}\in T_{p}M$ are linearly independent,
then we see $\Vert v_{1}\wedge \cdots \wedge v_{k} \Vert=\sqrt{\det (g(v_{i},v_{j}))}$.
Let $d_{M}$ be the length metric induced from $g$.
If $M$ is complete with respect to $d_{M}$,
then the Hopf-Rinow theorem for length spaces tells us that
the metric space $(M,d_{M})$ is a proper geodesic space.

For $x\in \bm$,
and the tangent space $T_{x}\bm$ at $x$ on $\bm$,
let $T_{x}^{\perp} \bm$ be the orthogonal complement of $T_{x}\bm$ in the tangent space at $x$ on $M$.
Take $u\in T_{x}^{\perp}\bm$.
For the second fundamental form $S$ of $\bm$,
let $A_{u}:T_{x}\bm \to T_{x}\bm$ be the \textit{shape operator} for $u$ defined as
\begin{equation*}
g(A_{u}v,w):=g(S(v,w),u).
\end{equation*}
Let $u_{x}\in T_{x}^{\perp}\bm$ denote the unit inner normal vector at $x$.
The \textit{mean curvature} $H_{x}$ at $x$ is defined by
\begin{equation*}
H_{x}:=\frac{1}{n-1} \tr A_{u_{x}}.
\end{equation*}
For the normal tangent bundle $\tbp:=\bigcup_{x\in \bm} T_{x}^{\perp}\bm$ of $\bm$,
let $0(\tbp)$ be the zero-section $\bigcup_{x\in \bm} \{\,0_{x}\in T_{x}^{\perp}\partial M\, \}$ of $T^{\perp}\bm$.
For $r>0$,
we put 
\begin{equation*}
U_{r}(0(\tbp)):=\bigcup_{x\in \bm} \{\,t\,u_{x}\in T_{x}^{\perp}\bm\mid t\in [0,r) \,\}.
\end{equation*}
For $x\in \bm$,
we denote by $\gamma_{x}:[0,T)\to M$ the normal geodesic with initial conditions $\gamma_{x}(0)=x$ and $\gamma_{x}'(0)=u_{x}$.
Note that
$\gamma_{x}$ is a normal geodesic in the usual sense in Riemannian geometry.
On an open neighborhood of $0(\tbp)$ in $\tbp$, 
the normal exponential map $\expp$ of $\bm$ is defined as follows:
For $x\in \bm$ and $u\in T_{x}^{\perp}\bm$,
put $\expp(x,u):=\gamma_{x}(\Vert u\Vert)$.
Since the boundary $\bm$ is smooth,
there exists an open neighborhood $U$ of $\bm$ satisfying the following:
(1) the map $\expp|_{(\expp)^{-1}(U\setminus \bm)}$ is a diffeomorphism onto $U\setminus \bm$;
(2) for every $p\in U$,
there exists a unique point $x\in \bm$ such that $d_{M}(p,x)=d_{M}(p,\bm)$;
in this case,
$\gamma_{x}|_{[0,d_{M}(p,\bm)]}$ is a unique normal minimal geodesic in $M$ from $x$ to $p$.
We call such an open set $U$ a \textit{normal neighborhood of} $\bm$.
If $\bm$ is compact,
then for some $r>0$,
the set $U_{r}(\bm)$ is a normal neighborhood of $\bm$.

We say that a Jacobi field $Y$ along $\gamma_{x}$ is a $\bm$-\textit{Jacobi field} if $Y$ satisfies the following initial conditions:
\begin{equation*}
Y(0)\in T_{x}\bm, \quad Y'(0)+A_{u_{x}}Y(0)\in T_{x}^{\perp}\bm.
\end{equation*}
We say that $\gamma_{x}(t_{0})$ is a \textit{conjugate point} of $\bm$ along $\gamma_{x}$
if there exists a non-zero $\bm$-Jacobi field $Y$ along $\gamma_{x}$ with $Y(t_{0})=0$.
Let $\tau_{1}(x)$ denote the first conjugate value for $\bm$ along $\gamma_{x}$.
It is well-known that for all $x\in \bm$ and $t>\tau_{1}(x)$,
we have $t>d_{M}(\gamma_{x}(t),\bm)$.

For all $x\in \bm$ and $t\in [0,\tau_{1}(x))$,
we denote by $\theta(t,x)$ the absolute value of the Jacobian of $\expp$ at $(x,tu_{x})\in \tbp$.
For each $x\in \bm$,
we choose an orthonormal basis $\{e_{x,i}\}_{i=1}^{n-1}$ of $T_{x}\bm$.
For each $i=1,\dots,n-1$,
let $Y_{x,i}$ be the $\bm$-Jacobi field along $\gamma_{x}$ with initial conditions $Y_{x,i}(0)=e_{x,i}$ and $Y'_{x,i}(0)=-A_{u_{x}}e_{x,i}$.
Note that for all $x\in \bm$ and $t\in [0,\tau_{1}(x))$,
we have $\theta(t,x)=\Vert Y_{x,1}(t)\wedge \cdots \wedge Y_{x,n-1}(t)\Vert$.
This does not depend on the choice of the orthonormal basis.

\subsection{Distance rigidity and metric rigidity}\label{sec:metric}
For $i=1,2$,
let $M_{i}$ be connected Riemannian manifolds with boundary with Riemannian metric $g_{i}$.
For each $i$,
the boundary $\bm_{i}$ carries the induced Riemannian metric $h_{i}$.
\begin{defi}
We say that a homeomorphism $\Phi:M_{1}\to M_{2}$ is a \textit{Riemannian isometry with boundary} from $M_{1}$ to $M_{2}$ if $\Phi$ satisfies the following conditions:
\begin{enumerate}
 \item $\Phi|_{\inte M_{1}}:\inte M_{1} \to \inte M_{2}$ is smooth, and $(\Phi|_{\inte M_{1}})^{\ast} (g_{2})=g_{1}$;\label{enum:inner isom}
 \item $\Phi|_{\bm_{1}}:\bm_{1} \to \bm_{2}$ is smooth, and $(\Phi|_{\bm_{1}})^{\ast} (h_{2})=h_{1}$.\label{enum:bdry isom}
\end{enumerate}
\end{defi}
If there exists a Riemannian isometry $\Phi:M_{1}\to M_{2}$ with boundary,
then the inverse $\Phi^{-1}$ is also a Riemannian isometry with boundary.

The following is well-known for manifolds without boundary (see e.g., Theorem 11.1 in \cite{H}).
\begin{lem}\label{lem:isometry without boundary}
Let $M$ and $N$ be connected Riemannian manifolds (without boundary) with Riemannian metric $g_{M}$ and with $g_{N}$,
respectively.
Let $d_{M}$ and $d_{N}$ be the Riemannian distances on $M$ and on $N$,
respectively.
Suppose that
a map $\Psi:M\to N$ is an isometry between the metric spaces $(M,d_{M})$ and $(N,d_{N})$.
Then $\Psi$ is smooth,
and $\Psi^{\ast} g_{N}=g_{M}$.
Namely,
$\Psi$ is a Riemannian isometry from $(M,g_{M})$ to $(N,g_{N})$.
\end{lem}
For manifolds with boundary,
we show the following:
\begin{lem}\label{lem:isometry with with boundary}
For $i=1,2$,
let $M_{i}$ be connected Riemannian manifolds with boundary with Riemannian metric $g_{i}$.
Then there exists a Riemannian isometry with boundary from $M_{1}$ to $M_{2}$
if and only if the metric space $(M_{1},d_{M_{1}})$ is isometric to $(M_{2},d_{M_{2}})$.
\end{lem}
\begin{proof}
For $i=1,2$,
we denote by $\Vert \cdot \Vert_{g_{i}}$ and by $\Vert \cdot \Vert_{h_{i}}$ the standard norms induced from $g_{i}$ and from $h_{i}$,
respectively.
For a piecewise smooth curve $\gamma$ in $M_{i}$,
we denote by $L_{g_{i}}(\gamma)$ the length of $\gamma$ induced from $g_{i}$.

First,
we show that
if $\Phi:M_{1}\to M_{2}$ is a Riemannian isometry with boundary,
then it is an isometry between the metric spaces $(M_{1},d_{M_{1}})$ and $(M_{2},d_{M_{2}})$.
It suffices to show that
$\Phi$ is a $1$-Lipschitz map from $(M_{1},d_{M_{1}})$ to $(M_{2},d_{M_{2}})$.
Pick $p,q\in M_{1}$.
Take $\epsilon>0$.
There exists a piecewise smooth curve $\gamma:[0,l]\to M_{1}$ such that $L_{g_{1}}(\gamma)<d_{M_{1}}(p,q)+\epsilon$.
Assume that
$\gamma$ is smooth at $t\in [0,l]$.
If $\gamma(t)$ belongs to $\inte M_{1}$,
then $\Vert (\Phi \circ \gamma)'(t)\Vert_{g_{2}}$ is equal to $\Vert \gamma'(t)\Vert_{g_{1}}$.
If $\gamma(t)$ belongs to $\bm_{1}$,
then $\Vert (\Phi \circ \gamma)'(t)\Vert_{h_{2}}$ is equal to $\Vert \gamma'(t)\Vert_{h_{1}}$,
and hence $L_{g_{2}}(\Phi \circ \gamma)$ is equal to $L_{g_{1}}(\gamma)$.
We have $d_{M_{2}}(\Phi(p),\Phi(q))<d_{M_{1}}(p,q)+\epsilon$.
This implies that
$\Phi$ is $1$-Lipschitz.

Next,
we show that
if $\Psi:M_{1}\to M_{2}$ is an isometry between the metric spaces $(M_{1},d_{M_{1}})$ and $(M_{2},d_{M_{2}})$,
then it is a Riemannian isometry with boundary.
To do this,
we first show that
$\Psi|_{\inte M_{1}}:\inte M_{1} \to \inte M_{2}$ is smooth,
and $(\Psi|_{\inte M_{1}})^{\ast} (g_{2})=g_{1}$.
Take $p\in \inte M_{1}$.
There exists a sufficiently small $r\in (0,\infty)$ such that
$U_{r}(p)$ and $U_{r}(\Psi(p))$ are strongly convex in $(\inte M_{1},g_{1})$ and in $(\inte M_{2},g_{2})$,
respectively.
Then $\Psi|_{U_{r}(p)}$ becomes an isometry between the metric subspaces $U_{r}(p)$ and $U_{r}(\Psi(p))$.
Applying Lemma \ref{lem:isometry without boundary} to the open Riemannian submanifolds $U_{r}(p)$ and $U_{r}(\Psi(p))$,
we see that
$\Psi|_{U_{r}(p)}$ is a smooth Riemannian isometry.
This implies that
$\Psi|_{\inte M_{1}}:\inte M_{1} \to \inte M_{2}$ is smooth,
and $(\Psi|_{\inte M_{1}})^{\ast} (g_{2})=g_{1}$.

We second show that
the map $\Psi|_{\bm_{1}}:\bm_{1} \to \bm_{2}$ is smooth,
and $(\Psi|_{\bm_{1}})^{\ast} (h_{2})=h_{1}$.
To do this,
we prove that
$\Psi|_{\bm_{1}}$ is an isometry between the metric spaces $(\bm_{1},d_{\bm_{1}})$ and $(\bm_{2},d_{\bm_{2}})$,
where $d_{\bm_{1}}$ and $d_{\bm_{2}}$ are the Riemannian distances on $\bm_{1}$ and on $\bm_{2}$,
respectively.
It suffices to show that
$\Psi|_{\bm_{1}}$ is a $1$-Lipschitz map from $(\bm_{1},d_{\bm_{1}})$ to $(\bm_{2},d_{\bm_{2}})$.
Take $x,y\in \bm_{1}$.
For every $\epsilon>0$,
there exists a piecewise smooth curve $\gamma:[0,l]\to \bm_{1}$ such that $L_{h_{1}}(\gamma)<d_{\bm_{1}}(x,y)+\epsilon$.
Fix $t \in [0,l]$ at which
$\gamma$ is smooth.
Since $\Psi$ is an isometry between $(M_{1},d_{M_{1}})$ and $(M_{2},d_{M_{2}})$,
we have
\begin{align*}
\Vert \gamma'(t) \Vert_{h_{1}} &=\Vert \gamma'(t) \Vert_{g_{1}}=\lim_{\delta \to 0}\frac{d_{M_{1}}(\gamma(t),\gamma(t+\delta))}{\delta}\\
                                                &=\lim_{\delta \to 0}\frac{d_{M_{2}}((\Psi \circ \gamma)(t),(\Psi \circ \gamma)(t+\delta))}{\delta}.
\end{align*}
Since $\bm_{2}$ is smooth,
and since $h_{2}$ is induced from $g_{2}$,
for every $z_{0}\in \bm_{2}$
we have
\begin{equation*}
\lim_{z\to z_{0}}\, \frac{d_{\bm_{2}}(z_{0},z)}{d_{M_{2}}(z_{0},z)}=1,
\end{equation*}
where the limit is taken with respect to $d_{\bm_{2}}$.
Hence,
we have
\begin{equation*}
\lim_{\delta \to 0}\frac{d_{\bm_{2}}((\Psi \circ \gamma)(t),(\Psi \circ \gamma)(t+\delta))}{d_{M_{2}}((\Psi \circ \gamma)(t),(\Psi \circ \gamma)(t+\delta))}=1;
\end{equation*}
in particular,
\begin{equation*}
\Vert \gamma'(t) \Vert_{h_{1}}=\lim_{\delta \to 0}\frac{d_{\bm_{2}}((\Psi \circ \gamma)(t),(\Psi \circ \gamma)(t+\delta))}{\delta}.
\end{equation*}
It follows that
\begin{equation*}
L_{h_{1}}(\gamma)=\int^{l}_{0}\,\lim_{\delta \to 0}\frac{d_{\bm_{2}}((\Psi \circ \gamma)(t),(\Psi \circ \gamma)(t+\delta))}{\delta}\,dt.
\end{equation*}
The right hand side coincides with the length of $\Psi \circ \gamma$ with respect to $d_{\bm_{2}}$ (see e.g., Section 2.7 in \cite{BBI}),
and is greater than or equal to $d_{\bm_{2}}(\Psi(x),\Psi(y))$.
Therefore,
$d_{\bm_{2}}(\Psi(x),\Psi(y))<d_{\bm_{1}}(x,y)+\epsilon$.
This implies that
$\Psi|_{\bm_{1}}$ is $1$-Lipschitz.
Thus,
we conclude that
$\Psi|_{\bm_{1}}$ is an isometry between $(\bm_{1},d_{\bm_{1}})$ and $(\bm_{2},d_{\bm_{2}})$.
Applying Lemma \ref{lem:isometry without boundary} to $\bm_{1}$ and $\bm_{2}$,
we see that
$\Psi|_{\bm_{1}}$ is smooth,
and $(\Psi|_{\bm_{1}})^{\ast} (h_{2})=h_{1}$.

This completes the proof of Lemma \ref{lem:isometry with with boundary}.
\end{proof}
\subsection{Comparison theorem}
For $\kappa\in \mathbb{R}$, 
let $s_{\kappa}(t)$ be a unique solution of the so-called Jacobi-equation $f''(t)+\kappa f(t)=0$ with initial conditions $f(0)=0$ and $f'(0)=1$.

The \textit{Laplacian} $\Delta$ of a smooth function on a Riemannian manifold is defined by
the minus of the trace of its Hessian.

It is well-known that
we have the following Laplacian comparison theorem for the distance function from a single point (see e.g., Proposition 3.6 in \cite{S}).
\begin{lem}\label{lem:point Laplacian}
Let $M$ be an $n$-dimensional,
connected complete Riemmanian manifold with boundary such that $\ric_{M}\geq (n-1)\kappa$.
Take $p\in \inte M$ and $u\in U_{p}M$.
Let $\rho_{p}:M\to \mathbb{R}$ be the function defined as $\rho_{p}(q):=d_{M}(p,q)$,
and let $\gamma_{u}:[0,t_{0})\to M$ be the normal minimal geodesic with initial conditions $\gamma_{u}(0)=p$ and $\gamma'_{u}(0)=u$
such that $\gamma_{u}$ lies in $\inte M$.
Then for all $t\in (0,t_{0})$, we have
\begin{equation*}
\Delta \rho_{p}(\gamma_{u}(t))\geq -(n-1)\frac{s_{\kappa}'(t)}{s_{\kappa}(t)}.
\end{equation*}
\end{lem}

\section{Cut locus for the boundary}\label{sec:Cut locus for the boundary}
Let $M$ be a connected complete Riemannian manifold with boundary with Riemannian metric $g$.
\subsection{Foot points}
For a point $p\in M$, 
we call $x\in \bm$ a \textit{foot point} on $\bm$ of $p$ if $d_{M}(p,x)=d_{M}(p,\bm)$.
Since $(M,d_{M})$ is proper, 
every point in $M$ has at least one foot point on $\bm$.
\begin{lem}\label{lem:foot vector}
For $p\in \inte M$, 
let $x\in \bm$ be a foot point on $\bm$ of $p$.
Then there exists a unique normal minimal geodesic $\gamma:[0,l]\to M$ from $x$ to $p$
such that $\gamma=\gamma_{x}|_{[0,l]}$,
where $l=\rho_{\bm}(p)$.
In particular,
$\gamma'(0)=u_{x}$ and $\gamma|_{(0,l]}$ lies in $\inte M$.
\end{lem}
\begin{proof}
Since $(M,d_{M})$ is a geodesic space,
there exists a normal minimal geodesic $\gamma:[0,l]\to M$ from $x$ to $p$.
Since $x$ is a foot point on $\bm$ of $p$,
we see that $\gamma|_{(0,l]}$ lies in $\inte M$.
We take a normal neighborhood $U$ of $\bm$.
If $p\in U\setminus \bm$, 
then $x$ is a unique foot point on $\bm$ of $p$,
and $\gamma=\gamma_{x}|_{[0,l]}$;
in particular, we have $\gamma'(0)=u_{x}$.
Even if $p\notin U\setminus \bm$, 
then for every sufficiently small $t>0$,
we see that $x$ is the foot point on $\bm$ of $\gamma(t)$.
Hence, $\gamma'(0)=u_{x}$.
This implies $\gamma=\gamma_{x}|_{[0,l]}$.
\end{proof}

\subsection{Cut locus}
Let $\tau:\bm\to \mathbb{R}\cup \{\infty\}$ be the function defined as
\begin{equation*}
\tau(x):=\sup \{\,t>0 \mid \rho_{\bm}(\gamma_{x}(t))=t\,\}.
\end{equation*}
Recall that
for all $x\in \bm$ and $t>\tau_{1}(x)$,
we have $t>\rho_{\bm}(\gamma_{x}(t))$.
Therefore,
for all $x\in \bm$,
we have $0<\tau(x)\leq \tau_{1}(x)$.

To study the cut locus,
we show the following:
\begin{lem}\label{lem:conti}
The function $\tau$ is continuous on $\bm$.
\end{lem}
\begin{proof}
Assume $x_{i}\to x$ in $\bm$.
First,
we show the upper semi-continuity of $\tau$.
We assume $\limsup_{i\to \infty} \tau(x_{i})<\infty$.
Take a subsequence $\{\tau(x_{j})\}$ of $\{\tau(x_{i})\}$ with $\tau(x_{j})\to \limsup_{i\to \infty} \tau(x_{i})$ as $j\to \infty$.
Put $p_{j}:=\gamma_{ x_{j} }( \tau(x_{j}) )$ and $p:=\gamma_{x}( \limsup_{i\to \infty} \tau(x_{i}) )$.
Since geodesics in $(\inte M,g)$ depend continuously on the initial direction and the parameter, 
we see $p_{j}\to p$ in $M$ as $j\to \infty$.
By the definition of $\tau$,
for all $j$ we have $\rho_{\bm}(p_{j})=\tau(x_{j})$.
By letting $j\to \infty$, 
we obtain $\rho_{\bm}(p)=\limsup_{i\to \infty} \tau(x_{i})$.
Hence,
$\limsup_{i\to \infty} \tau(x_{i})\leq \tau(x)$.
In a similar way,
we see that
if $\limsup_{i\to \infty} \tau(x_{i})=\infty$,
then $\tau(x)=\infty$.
Therefore,
we have shown the upper semi-continuity.

Next,
we show the lower semi-continuity of $\tau$.
We may assume $\liminf_{i\to \infty} \tau(x_{i})<\infty$.
The proof is done by contradiction.
We suppose $\liminf_{i\to \infty} \tau(x_{i})<\tau(x)$.
Choose $\delta>0$ such that $\liminf_{i\to \infty} \tau(x_{i})+\delta<\tau(x)$.
Take a subsequence $\{\tau(x_{j})\}$ of $\{\tau(x_{i})\}$ with $\tau(x_{j})\to \liminf_{i\to \infty} \tau(x_{i})$ as $j\to \infty$.
By the definition of $\tau$, 
we have $\tau(x_{j})+\delta>d_{M}(\gamma_{x_{j}}(\tau(x_{j})+\delta),\bm)$.
Since $\gamma_{x_{j}}(\tau(x_{j})+\delta) \to \gamma_{x}(\liminf_{i\to \infty} \tau(x_{i})+\delta)$ in $M$,
we have
\begin{equation*}
\liminf_{i\to \infty} \tau(x_{i})+\delta>\rho_{\bm}(\gamma_{x}(\liminf_{i\to \infty} \tau(x_{i})+\delta)).
\end{equation*}
On the other hand, 
$\liminf_{i\to \infty} \tau(x_{i})+\delta<\tau(x)$.
This contradicts the definition of $\tau$.
Hence,
we have shown the lower semi-continuity.
\end{proof}
By Lemma \ref{lem:foot vector},
we have the following:
\begin{lem}\label{lem:shinitai}
For all $r>0$,
we have
\begin{equation*}
B_{r}(\bm)=\expp \left(\bigcup_{x\in \bm} \{tu_{x}\mid t\in [0,\min\{r,\tau(x)\}] \}\right).
\end{equation*}
\end{lem}
\begin{proof}
Take $p\in B_{r}(\bm)$,
and let $x$ be a foot point on $\bm$ of $p$.
By Lemma \ref{lem:foot vector},
there exists a unique normal minimal geodesic $\gamma:[0,l]\to M$ from $x$ to $p$ such that $\gamma=\gamma_{x}|_{[0,l]}$,
where $l=\rho_{\bm}(p)$.
Since $x$ is a foot point on $\bm$ of $p$,
we have $l\leq r$,
and $l\leq \tau(x)$.
Hence,
\begin{equation*}
B_{r}(\bm)\subset \expp \left(\bigcup_{x\in \bm} \{tu_{x}\mid t\in [0,\min\{r,\tau(x)\}] \}\right).
\end{equation*}

On the other hand,
take $x\in \bm$ and $t\in [0,\min\{r,\tau(x)\}]$.
By the definition of $\tau$,
the point $x$ is a foot point on $\bm$ of $\gamma_{x}(t)$.
Therefore,
$\rho_{\bm}(\gamma_{x}(t))=t\leq r$.
This implies the opposite inclusion.
\end{proof}
For the inscribed radius $D(M,\bm)$ of $M$,
from the definition of $\tau$,
it follows that $\sup_{x\in \bm} \tau(x)\leq D(M,\bm)$.
Lemma \ref{lem:foot vector} implies the opposite.
Hence, 
we have $D(M,\bm)=\sup_{x\in \bm} \tau(x)$.

We put
\begin{align*}
TD_{\bm}  &:= \bigcup_{x\in \bm} \{\,t\,u_{x} \in T^{\perp}_{x}\bm \mid t\in[0,\tau(x)) \,\},\\
T\cut \bm &:= \bigcup_{x\in \bm} \{\,\tau(x)\,u_{x} \in T^{\perp}_{x}\bm \mid \tau(x)<\infty \,\},
\end{align*}
and define $D_{\bm}:=\expp (TD_{\bm})$ and $\cut \bm:=\expp (T\cut \bm)$.
We call $\cut \bm$ the \textit{cut locus for the boundary} $\bm$.
By Lemma \ref{lem:foot vector},
we have $\inte M=(D_{\bm}\setminus \bm)\cup \cut \bm$ and $M=D_{\bm}\cup \cut \bm$.

The continuity of $\tau$ tells us the following:
\begin{lem}\label{lem:compact}
Suppose $\bm$ is compact.
Then $D(M,\bm)<\infty$ if and only if $M$ is compact.
\end{lem}
\begin{proof}
If $D(M,\bm)<\infty$,
then $\sup_{x\in \bm} \tau(x)<\infty$.
By the continuity of $\tau$,
the set $TD_{\bm}\cup T\cut \bm$ is closed in $T^{\perp}\bm$.
Since $\bm$ is compact,
the set is compact in $T^{\perp}\bm$.
The set $D_{\bm}\cup \cut \bm$ coincides with $M$.
The continuity of $\expp|_{TD_{\bm}\cup T\cut \bm}$ implies that $M$ is compact.
On the other hand,
if $M$ is compact,
then the function $\rho_{\bm}$ is finite on $M$;
in particular,
$D(M,\bm)<\infty$.
\end{proof}
Furthermore,
we have:
\begin{prop}\label{prop:vol of cut}
$\vol_{g} \cut \bm=0$.
\end{prop}
\begin{proof}
By Lemma \ref{lem:conti},
and by the Fubini theorem,
the graph
\begin{equation*}
\{\,(x,\tau(x)) \mid x\in \bm,\,\tau(x)<\infty \,\}
\end{equation*}
of $\tau$ is a null set of $\bm\times [0,\infty)$.
A map $\Psi:\bm \times [0,\infty)\to T^{\perp}\bm$ defined by $\Psi(x,t):=(x,t u_{x})$ is smooth.
In particular,
the set $T\cut \bm$ is also a null set of $T^{\perp}\bm$.
By the definition of $\tau$,
the set $\cut \bm$ is contained in $\inte M$.
Hence,
$\expp$ is smooth on an open neighborhood of $T\cut \bm$ in $T^{\perp}\bm$.
Therefore,
we see $\vol_{g} \cut \bm=0$.
\end{proof}
We next show the following characterization of $\tau$:
\begin{lem}\label{lem:cut}
Let $T>0$.
Take $x\in \bm$ with $\tau(x)<\infty$.
Then $T=\tau(x)$ if and only if $T=\rho_{\bm}(\gamma_{x}(T))$, 
and at least one of the following holds:
\begin{enumerate}
 \item $\gamma_{x}(T)$ is the first conjugate point of $\bm$ along $\gamma_{x}$;\label{enum:first conjugate}
 \item there exists a foot point $y\in \bm\setminus \{x\}$ on $\bm$ of $\gamma_{x}(T)$.\label{enum:distinct foot}
\end{enumerate}
\end{lem}
\begin{proof}
First,
we assume $T=\rho_{\bm}(\gamma_{x}(T))$.
By the definition of $\tau$,
we have $T\leq \tau(x)$.
If (\ref{enum:first conjugate}) holds, 
then $T$ is equal to $\tau_{1}(x)$;
in particular,
$T=\tau(x)$.
Suppose that (\ref{enum:distinct foot}) holds.
We assume $T<\tau(x)$, 
and take $\delta>0$ such that $T+\delta<\tau(x)$.
If $\gamma'_{x}(T)=-\gamma'_{y}(T)$ at $\gamma_{x}(T)$, 
then $\gamma_{x}(T+\delta)=\gamma_{y}(T-\delta)$.
Since $T\leq \tau(y)$, we have
\begin{equation*}
\rho_{\bm}(\gamma_{x}(T+\delta))=\rho_{\bm}(\gamma_{y}(T-\delta))=T-\delta.
\end{equation*}
This is in contradiction with $T+\delta<\tau(x)$.
If $\gamma'_{x}(T)\neq -\gamma'_{y}(T)$ at $\gamma_{x}(T)$, then for all $t\in (T,T+\delta]$, we have
\begin{equation*}
\rho_{\bm}(\gamma_{x}(t))<d_{M}(\gamma_{x}(t),\gamma_{x}(T))+d_{M}(\gamma_{x}(T),y)\leq t.
\end{equation*}
This contradicts $t\leq T+\delta<\tau(x)$.
Hence,
we see $T=\tau(x)$.

Next,
we assume $T=\tau(x)$.
Then we have $T=\rho_{\bm}(\gamma_{x}(T))$.
Put $p:=\gamma_{x}(T)$.
Assuming that $p$ is not the first conjugate point of $\bm$ along $\gamma_{x}$,
we will prove (\ref{enum:distinct foot}).
Take an open neighborhood $\bar{U}$ of $(x,Tu_{x})$ in $T^{\perp}\bm$ such that $\expp|_{\bar{U}}:\bar{U}\to \expp(\bar{U})$ is a diffeomorphism.
Put $U:=\expp(\bar{U})$.
For every sufficiently large $i\in \mathbb{N}$, 
we put $p_{i}:=\gamma_{x}(T+1/i)$,
and take a foot point $x_{i}$ on $\bm$ of $p_{i}$.
By Lemma \ref{lem:foot vector},
there exists a unique normal minimal geodesic $\gamma_{i}:[0,l_{i}]\to M$ from $x_{i}$ to $p_{i}$ such that $\gamma_{i}=\gamma_{x_{i}}|_{[0,l_{i}]}$,
where $l_{i}=\rho_{\bm}(p_{i})$.
Since $(M,d_{M})$ is proper,
by taking a subsequence if necessary,
we may assume that
for some $y\in \bm$,
we have $x_{i}\to y$ in $\bm$.
Since $x_{i}$ is a foot point on $\bm$ of $p_{i}$ and $p_{i}\to p$ in $M$,
we see that $y$ is a foot point on $\bm$ of $p$.
If $x=y$,
then for every sufficiently large $i\in \mathbb{N}$, 
we have $(x_{i},l_{i}\,u_{x_{i}})\in \bar{U}$ and $\expp(x,(T+ 1/i)\,u_{x})=\expp (x_{i},l_{i}\,u_{x_{i}})$.
By the injectivity of $\expp|_{\bar{U}}$,
we have $T+ 1/i=l_{i}$.
This is in contradiction with $T+1/i>l_{i}$.
Hence,
we see $x\neq y$.
This completes the proof.
\end{proof}
From Lemma \ref{lem:cut},
we derive the following:
\begin{lem}\label{lem:cut and regular}
We have $\cut \bm \cap D_{\bm}=\emptyset$.
In particular, 
\begin{equation*}
\inte M=(D_{\bm}\setminus \bm) \sqcup \cut \bm,\quad M=D_{\bm}\sqcup \cut \bm.
\end{equation*}
\end{lem}
\begin{proof}
Suppose that
there exists $p\in \cut \bm \cap D_{\bm}$.
Then there exist $x,\,y\in \bm$ and $l\in (0,\tau(y))$ such that $p=\gamma_{x}(\tau(x))=\gamma_{y}(l)$.
By the definition of $\tau$,
we have $l=\tau(x)$;
in particular,
$x\neq y$.
Furthermore,
by the definition of $\tau$,
we see that $x$ and $y$ are foot points on $\bm$ of $p$.
By Lemma \ref{lem:cut},
we have $l=\tau(y)$.
This is a contradiction.
Therefore,
we have $\cut \bm \cap D_{\bm}=\emptyset$.
Since $\inte M=(D_{\bm}\setminus \bm)\cup  \cut \bm$ and $M=D_{\bm}\cup \cut \bm$,
we prove the lemma.
\end{proof}
For the connectedness of the boundary,
we show:
\begin{lem}\label{lem:shinitai4}
If $\cut \bm=\emptyset$,
then $\bm$ is connected.
\end{lem}
\begin{proof}
Suppose that $\bm$ is not connected.
Let $\{\partial M_{i}\}_{i\geq 2}$ be the connected components of $\bm$.
By Lemma \ref{lem:cut},
for every $p\in D_{\bm}\setminus \bm$,
there exists a unique foot point on $\bm$ of $p$.
For each $i$,
we denote by $D_{\bm_{i}}$ the set of all points in $D_{\bm}\setminus \bm$ whose foot points are contained in $\bm_{i}$.
By the continuity of $\tau$,
the sets $D_{\bm_{i}}\setminus \bm,\, i\geq 2$, are mutually disjoint domains in $\inte M$.
Lemma \ref{lem:cut and regular} implies that
$\inte M$ coincides with $(\bigsqcup_{i\geq 2}D_{\bm_{i}})\sqcup \cut \bm$.
Since $\cut\bm=\emptyset$,
the set $\inte M$ is not connected.
This is a contradiction.
\end{proof}
By the continuity of $\tau$,
the set $TD_{\bm}\setminus 0(T^{\perp}\bm)$ is a domain in $T^{\perp}\bm$.
Using Lemma \ref{lem:cut},
we see the following:
\begin{lem}\label{lem:regular and injective}
$TD_{\bm}\setminus 0(T^{\perp}\bm)$ is a maximal domain in $T^{\perp}\bm$ on which $\expp$ is regular and injective.
\end{lem}
We show the smoothness of $\rho_{\bm}$ on the set $\inte M\setminus \cut \bm$.
\begin{prop}\label{prop:distance function}
The function $\rho_{\bm}$ is smooth on $\inte M\setminus \cut \bm$.
Moreover,
for each $p\in \inte M\setminus \cut \bm$,
the gradient vector $\nabla \rho_{\bm}(p)$ of $\rho_{\bm}$ at $p$
is given by $\nabla \rho_{\bm}(p)=\gamma'(l)$,
where $\gamma:[0,l]\to M$ is the normal minimal geodesic from the foot point on $\bm$ of $p$ to $p$.
\end{prop}
\begin{proof}
By Lemma \ref{lem:regular and injective}, 
the map $\expp|_{TD_{\bm}\setminus 0(T^{\perp}\bm)}$ is a diffeomorphism onto $D_{\bm}\setminus \bm$.
Lemma \ref{lem:cut and regular} implies
$\inte M\setminus \cut \bm=D_{\bm}\setminus \bm$.
For all $q\in \inte M\setminus \cut \bm$, 
we have $\rho_{\bm}(q)=\Vert (\expp)^{-1}(q) \Vert$.
Hence, $\rho_{\bm}$ is smooth on $\inte M\setminus \cut \bm$.

For any vector $v\in T_{p}M$, 
we take a smooth curve $c:(-\epsilon,\epsilon)\to \inte M$ tangent to $v$ at $p=c(0)$.
We may assume $c(s)\in \inte M\setminus \cut \bm$ when $\vert s \vert$ is sufficiently small.
By Lemma \ref{lem:cut}, 
there exists a unique foot point $\bar{c}(s)$ on $\bm$ of $c(s)$.
By Lemma \ref{lem:foot vector},
we obtain a smooth variation of $\gamma$ by taking normal minimal geodesics in $M$ from $\bar{c}(s)$ to $c(s)$.
The first variation formula for the variation implies $(\rho_{\bm}\circ c)'(0)= g(v,\gamma'(l))$.
Therefore,
we have $\nabla \rho_{\bm}(p)=\gamma'(l)$.
\end{proof}

\section{Comparison theorems}\label{sec:Comparison theorems}
In this section,
we prove Theorem \ref{thm:volume comparison}.
\subsection{Basic comparison}
We refer to the following absolute comparison inequality that has been shown by Heintze and Karcher in Subsection 3.4 in \cite{HK}.
\begin{lem}[\cite{HK}]\label{lem:comp1}
Let $M$ be an $n$-dimensional, 
connected complete Riemannian manifold with boundary with Riemannian metric $g$.
Take a point $x\in \bm$.
Suppose that for all $t\in (0,\min \{\tau_{1}(x),\bar{C}_{\kappa,\lambda}\} )$,
we have $\ric_{g}(\gamma'_{x}(t)) \geq (n-1)\kappa$,
and suppose $H_{x}\geq \lambda$.
Then for all $t\in (0,\min \{\tau_{1}(x),\bar{C}_{\kappa,\lambda}\})$, we have
\begin{equation*}
\frac{\thetaex'(t,x)}{\thetaex(t,x)}\leq (n-1)\frac{s_{\kappa,\lambda}'(t)}{s_{\kappa,\lambda}(t)}.
\end{equation*}
\end{lem}
\begin{rem}\label{rem:comp eq}
In the case in Lemma \ref{lem:comp1},
we choose an orthonormal basis $\{e_{x,i}\}_{i=1}^{n-1}$ of $T_{x}\bm$,
and let $\{Y_{x,i}\}^{n-1}_{i=1}$ be the $\bm$-Jacobi fields along $\gamma_{x}$
with initial conditions $Y_{x,i}(0)=e_{x,i}$ and $Y_{x,i}'(0)=-A_{u_{x}}e_{x,i}$.
Then there exists $t_{0}\in (0,\min \{\tau_{1}(x),\bar{C}_{\kappa,\lambda}\})$ such that
\begin{equation*}
\frac{\thetaex'(t_{0},x)}{\thetaex(t_{0},x)}=(n-1)\frac{s_{\kappa,\lambda}'(t_{0})}{s_{\kappa,\lambda}(t_{0})}.
\end{equation*}
if and only if
for all $i=1,\dots,n-1$ and $t\in [0,t_{0}]$,
we have $Y_{x,i}(t)=s_{\kappa,\lambda}(t)\, E_{x,i}(t)$,
where $E_{x,i}$ are the parallel vector fields along $\gamma_{x}$ with initial condition $E_{x,i}(0)=e_{x,i}$ (see \cite{HK}).
\end{rem}
The following Laplacian comparison theorem has been stated by Kasue in Corollary 2.42 in \cite{K2}.
\begin{thm}[\cite{K2}]\label{thm:boundary Laplacian}
Let $M$ be an $n$-dimensional,
connected complete Riemannian manifold with boundary with Riemannian metric $g$.
Take $x\in \bm$.
Suppose that for all $t\in (0,\tau(x))$,
we have $\ric_{g}(\gamma'_{x}(t)) \geq (n-1)\kappa$,
and suppose $H_{x}\geq \lambda$.
Then for all $t\in (0,\tau(x))$, we have
\begin{equation*}
\Delta \rho_{\bm} (\gamma_{x}(t)) \geq -(n-1)\frac{s_{\kappa,\lambda}'(t)}{s_{\kappa,\lambda}(t)}.
\end{equation*}
\end{thm}
\begin{rem}\label{rem:lap eq}
In the case in Theorem \ref{thm:boundary Laplacian},
for all $t\in (0,\tau(x))$,
we have $\Delta \rho_{\bm} (\gamma_{x}(t))=-\thetaex'(t,x)/\thetaex(t,x)$.
Therefore,
the equality case in Theorem \ref{thm:boundary Laplacian}
results into that in Lemma \ref{lem:comp1} (see Remark \ref{rem:comp eq}).
\end{rem}
By Lemma \ref{lem:comp1},
we have the following relative comparison inequality.
\begin{lem}\label{lem:comp2}
Let $M$ be an $n$-dimensional, 
connected complete Riemannian manifold with boundary with Riemannian metric $g$.
Take a point $x\in \bm$.
Suppose that for all $t\in (0,\min \{\tau_{1}(x),\bar{C}_{\kappa,\lambda}\} )$,
we have $\ric_{g}(\gamma'_{x}(t)) \geq (n-1)\kappa$,
and suppose $H_{x}\geq \lambda$.
Then for all $s,t\in [0,\min \{\tau_{1}(x),\bar{C}_{\kappa,\lambda}\})$ with $s\leq t$,
\begin{equation*}
\frac{\thetaex(t,x)}{\thetaex(s,x)}\leq \frac{s_{\kappa,\lambda}^{n-1}(t)}{s_{\kappa,\lambda}^{n-1}(s)};
\end{equation*}
in particular,
$\theta(t,x)\leq s^{n-1}_{\kappa,\lambda}(t)$.
Moreover,
if $\kappa$ and $\lambda$ satisfy the ball-condition, 
then $\tau_{1}(x)\leq \const$.
\end{lem}
\begin{proof}
Take $\tilde{x}\in \partial M^{n}_{\kappa,\lambda}$.
By Lemma \ref{lem:comp1}, 
for all $t\in (0,\min \{\tau_{1}(x),\bar{C}_{\kappa,\lambda}\})$,
\begin{equation*}
\frac{d}{dt} \log \frac{\theta(t,\tilde{x}) }{\thetaex(t,x)}=\frac{\theta'(t,\tilde{x})}{\theta(t,\tilde{x})}-\frac{\thetaex'(t,x)}{\thetaex(t,x)}\geq 0.
\end{equation*}
Hence,
for all $s,t\in (0,\min \{\tau_{1}(x),\bar{C}_{\kappa,\lambda}\})$ with $s\leq t$,
we have
\begin{equation*}
\frac{\thetaex(t,x)}{\thetaex(s,x)}\leq \frac{\theta(t,\tilde{x})}{\theta(s,\tilde{x})}.
\end{equation*}
In the inequality,
by letting $s\to 0$,
we have $\thetaex(t,x)\leq \theta(t,\tilde{x})$.
Hence, 
for all $s,t\in [0,\min \{\tau_{1}(x),\bar{C}_{\kappa,\lambda}\})$ with $s\leq t$,
we have the desired inequality.

Let $\kappa$ and $\lambda$ satisfy the ball-condition.
We suppose $\const<\tau_{1}(x)$.
For all $t\in [0,\const)$, 
we have $\thetaex(t,x)\leq s^{n-1}_{\kappa,\lambda}(t)$.
By letting $t\to \const$, 
we have $\thetaex(\const,x)=0$.
Since $\const<\tau_{1}(x)$,
the point $\gamma_{x}(\const)$ is not a conjugate point of $\bm$ along $\gamma_{x}$.
Hence,
there exists a nonzero $\bm$-Jacobi field $Y$ along $\gamma_{x}$ such that $Y(\const)=0$;
in particular, 
$\gamma_{x}(\const)$ is a conjugate point of $\bm$ along $\gamma_{x}$.
This is a contradiction.
Therefore, 
we have $\tau_{1}(x)\leq \const$.
\end{proof}
\subsection{Inscribed radius comparison}
Using Lemma \ref{lem:comp2},
we will give a proof of the following lemma that has been already proved by Kasue in Theorem A in \cite{K3}.
\begin{lem}[\cite{K3}]\label{lem:shinitai2}
Let $\kappa\in \mathbb{R}$ and $\lambda \in \mathbb{R}$ satisfy the ball-condition.
Let $M$ be an $n$-dimensional, 
connected complete Riemannian manifold with boundary such that
$\ric_{M}\geq (n-1)\kappa$ and $H_{\bm}\geq \lambda$.
Then for all $x\in \bm$,
we have $\tau(x)\leq \const$;
in particular,
$\dm \leq \const$.
\end{lem}
\begin{proof}
Take $x\in \bm$.
By the definition of $\tau$,
the geodesic $\gamma_{x}|_{(0,\tau(x)]}$ lies in $\inte M$.
If $\const<\tau(x)$,
then by Lemma \ref{lem:comp2},
we see that $\gamma_{x}(\const)$ is a conjugate point of $\bm$ along $\gamma_{x}$.
We obtain $\tau_{1}(x)<\tau(x)$.
This contradicts the relation between $\tau$ and $\tau_{1}$.
Hence,
$\tau(x)\leq \const$.
\end{proof}
The following rigidity theorem has been proved in Theorem A in \cite{K3}.
\begin{thm}[\cite{K3}]\label{thm:ball rigid}
Let $\kappa\in \mathbb{R}$ and $\lambda \in \mathbb{R}$ satisfy the ball-condition.
Let $M$ be an $n$-dimensional, 
connected complete Riemannian manifold with boundary such that
$\ric_{M}\geq (n-1)\kappa$ and $H_{\bm}\geq \lambda$.
If there exists a point $p\in M$ such that $\rho_{\bm}(p)=\const$,
then the metric space $(M,d_{M})$ is isometric to $(\ball,d_{\ball})$.
\end{thm}
\subsection{Volume comparison}
By the coarea formula (see e.g., Theorem 3.2.3 in \cite{F}),
we have the following:
\begin{lem}\label{lem:vol}
Let $M$ be a connected complete Riemannian manifold with boundary with Riemannian metric $g$.
Suppose $\partial M$ is compact.
Let $r$ be a positive number such that
$U_{r}(\bm)$ is a normal neighborhood of $\bm$.
Then we have 
\begin{equation*}
\vol_{g} B_{r}(\bm)=\int_{\bm} \int^{r}_{0}\thetae(t,x)\,dt\,d\vol_{h}.
\end{equation*}
\end{lem}
From Lemma \ref{lem:vol},
we derive the following:
\begin{lem}\label{lem:vol2}
Let $M$ be a connected complete Riemannian manifold with boundary with Riemannian metric $g$.
Suppose $\partial M$ is compact.
Then for all $r>0$,
we have 
\begin{equation*}
\vol_{g} B_{r}(\bm)=\int_{\bm} \int^{\min\{r,\tau(x)\}}_{0}\thetae(t,x)\,dt\,d\vol_{h}.
\end{equation*}
\end{lem}
\begin{proof}
Take $r>0$.
By Lemma \ref{lem:shinitai},
we have
\begin{equation*}
B_{r}(\bm)=\expp \left(\bigcup_{x\in \bm} \{tu_{x}\mid t\in [0,\min\{r,\tau(x)\}] \}\right).
\end{equation*}
From Lemma \ref{lem:regular and injective},
it follows that
the map $\expp$ is diffeomorphic on $\bigcup_{x\in \bm} \{tu_{x}\mid t\in (0,\min\{r,\tau(x)\}) \}$.
Therefore,
by Proposition \ref{prop:vol of cut} and Lemma \ref{lem:vol}, 
we have the desired equality.
\end{proof}
We prove Theorem \ref{thm:volume comparison}.
\begin{proof}[Proof of Theorem \ref{thm:volume comparison}]
We define a function $\bthetae:[0,\infty) \times \bm \to \mathbb{R}$ by
\[
  \bthetae(t,x) := \begin{cases}
                            \thetae(t,x) & \text{if $t\leq \tau(x)$}, \\
                            0           & \text{if $t> \tau(x)$}.
                       \end{cases}
\]
By Lemma \ref{lem:vol2},
we have 
\begin{equation*}
\vol_{g} B_{r}(\bm)=\int_{\bm} \int^{r}_{0}\bthetae(t,x)\,dt\,d\vol_{h}.
\end{equation*}
Lemma \ref{lem:shinitai2} implies that
for each $x\in \bm$,
we have $\tau(x)\leq \bar{C}_{\kappa,\lambda}$.
Therefore,
from Lemma \ref{lem:comp2}, 
it follows that
for all $s,t\in [0,\infty)$ with $s\leq t$,
\begin{equation*}
\bthetae(t,x)\; \bar{s}_{\kappa,\lambda}^{n-1}(s) \leq \bthetae(s,x)\; \bar{s}_{\kappa,\lambda}^{n-1}(t).
\end{equation*}
Integrating the both sides of the above inequality over $[0,r]$ with respect to $s$,
and then doing that over $[r,R]$ with respect to $t$, we see
\begin{equation*}
\frac{\int^{R}_{r}\bthetae(t,x)\,dt}{\int^{r}_{0}\bthetae(s,x)\,ds}\leq \frac{\int^{R}_{r} \bar{s}_{\kappa,\lambda}^{n-1}(t)\,dt}{\int^{r}_{0} \bar{s}_{\kappa,\lambda}^{n-1}(s)\,ds}.
\end{equation*}
Hence,
we have
\begin{align*}
\frac{\vol_{g} B_{R}(\bm)}{\vol_{g} B_{r}(\bm)}&=    1+\frac{\int_{\bm} \int^{R}_{r}\bthetae(t,x)\,dt\,d\vol_{h}}{\int_{\bm} \int^{r}_{0}\bthetae(s,x)\,ds\,d\vol_{h}}\\
                                               &\leq 1+\frac{\int^{R}_{r}\bar{s}_{\kappa,\lambda}^{n-1}(t)\,dt}{\int^{r}_{0}\bar{s}_{\kappa,\lambda}^{n-1}(s)\,ds}
                                               =\frac{\vol B_{R}(\partial M^{n}_{\kappa,\lambda})}{\vol B_{r}(\partial M^{n}_{\kappa,\lambda})}.
\end{align*}
This completes the proof of Theorem \ref{thm:volume comparison}.
\end{proof}
\begin{rem}\label{rem:volume eq}
In the case in Theorem \ref{thm:volume comparison},
we suppose that there exists $R>0$ such that
for all $r\in (0,R]$, we have
\begin{equation*}
\frac{\vol_{g} B_{R}(\bm)}{\vol_{g} B_{r}(\bm)}=\frac{\vol B_{R}(\partial M^{n}_{\kappa,\lambda})}{\vol B_{r}(\partial M^{n}_{\kappa,\lambda})}.
\end{equation*}
In this case, 
for all $t\in (0,R]$ and $x\in \bm$, 
we have $\bthetae(t,x)=\bar{s}^{n-1}_{\kappa,\lambda}(t)$.
We choose an orthonormal basis $\{ e_{x,i} \}_{i=1}^{n-1}$ of $T_{x}\bm$.
Let $Y_{x,i}$ be the $\bm$-Jacobi field along $\gamma_{x}$
with initial conditions $Y_{x,i}(0)=e_{x,i}$ and $Y_{x,i}'(0)=-A_{u_{x}}e_{x,i}$.
For all $i=1,\dots,n-1$,
and for all $t\in [0,\min \{R,\bar{C}_{\kappa,\lambda}\}]$ and $x\in \bm$, 
we have $Y_{x,i}(t)=s_{\kappa,\lambda}(t)\, E_{x,i}(t)$,
where $E_{x,i}$ are the parallel vector fields along $\gamma_{x}$ with initial condition $E_{x,i}(0)=e_{x,i}$.
\end{rem}

\section{Volume growth rigidity}\label{sec:Volume growth rigidity}
\subsection{Volume growth}
By Theorem \ref{thm:volume comparison},
we have the following:
\begin{prop}\label{prop:volume growth sup}
Let $M$ be an $n$-dimensional,
connected complete Riemannian manifold with boundary with Riemannian metric $g$ such that $\ric_{M}\geq (n-1)\kappa$ and $H_{\bm} \geq \lambda$.
Suppose $\partial M$ is compact.
Let $h$ denote the induced Riemannian metric on $\bm$.
Then
\begin{equation*}
\limsup_{r\to \infty} \frac{\vol_{g}B_{r}(\bm)}{f_{n,\kappa,\lambda}(r)} \leq \vol_{h}\bm.
\end{equation*}
\end{prop}
\begin{proof}
Take $r>0$.
By Lemma \ref{lem:vol2},
we have
\begin{equation*}
\vol_{g} B_{r}(\bm)=\int_{\bm} \int^{\min\{r,\tau(x)\}}_{0}\thetae(t,x)\,dt\,d\vol_{h}.
\end{equation*}
By Lemma \ref{lem:comp2},
for all $x\in \bm$ and $t\in (0,\min\{r,\tau(x)\})$,
we have $\thetae(t,x)\leq s^{n-1}_{\kappa,\lambda}(t)$.
Integrating the both sides of the inequality over $(0,\min\{r,\tau(x)\})$ with respect to $t$,
and then doing that over $\bm$ with respect to $x$, 
we see $\vol_{g} B_{r}(\bm)/f_{n,\kappa,\lambda}(r)\leq \vol_{h}\bm$.
Letting $r\to \infty$,
we obtain the desired inequality.
\end{proof}
\subsection{Volume growth rigidity}
In the equality case in Theorem \ref{thm:volume comparison},
$\tau$ satisfies the following property:
\begin{lem}\label{lem:half rigid}
Let $M$ be an $n$-dimensional, 
connected complete Riemannian manifold with boundary with Riemannian metric $g$
such that $\ric_{M}\geq (n-1)\kappa$ and $H_{\bm} \geq \lambda$.
Suppose $\partial M$ is compact.
Assume that
there exists $R\in (0,\bar{C}_{\kappa,\lambda}]\setminus \{\infty\}$ such that
for all $r\in (0,R]$,
we have
\begin{equation*}
\frac{\vol_{g} B_{R}(\bm)}{\vol_{g} B_{r}(\bm)}= \frac{\vol_{g^{n}_{\kappa,\lambda}} B_{R}(\partial M^{n}_{\kappa,\lambda})}{\vol_{g^{n}_{\kappa,\lambda}} B_{r}(\partial M^{n}_{\kappa,\lambda})}.
\end{equation*}
Then for all $x\in \bm$,
we have $\tau(x)\geq R$.
\end{lem}
\begin{proof}
Suppose that for some $x_{0}\in \bm$,
we have $\tau(x_{0})<R$.
Put $t_{0}:=\tau(x_{0})$.
Take $\epsilon>0$ with $t_{0}+\epsilon<R$.
By the continuity of $\tau$,
there exists a closed geodesic ball $B$ in $\bm$ centered at $x_{0}$ such that
for all $x\in B$,
we have $\tau(x)\leq t_{0}+\epsilon$.
By Lemmas \ref{lem:comp2} and \ref{lem:vol2},
we see that $\vol_{g} B_{R}(\bm)$ is not larger than
\begin{equation*}
\int_{\bm\setminus B}\,\int^{\min\{R,\tau(x)\}}_{0}s^{n-1}_{\kappa,\lambda}(t)\,dt\,d\vol_{h}
+\int_{B}\,\int^{t_{0}+\epsilon}_{0}s^{n-1}_{\kappa,\lambda}(t)\,dt\,d\vol_{h}.
\end{equation*}
This is smaller than $(\vol_{h}\bm)\; f_{n,\kappa,\lambda}(R)$.
On the other hand,
by the assumption,
we see that $f_{n,\kappa,\lambda}(R)$ is equal to $\vol_{g} B_{R}(\bm)/\vol_{h}\bm$.
This is a contradiction.
\end{proof}
In the case in Lemma \ref{lem:half rigid},
for every $r\in (0,R)$,
the level set $\rho_{\bm}^{-1}(r)$ is an $(n-1)$-dimensional submanifold of $M$.
In particular,
$(B_{r}(\bm),g)$ is an $n$-dimensional (not necessarily, connected) complete Riemannian manifold with boundary.
We denote by $d_{B_{r}(\bm)}$ and by $d_{\kappa,\lambda,r}$
the Riemannian distances on $(B_{r}(\bm),g)$ and on $[0,r]\times_{\kappa,\lambda}\bm$,
respectively.
\begin{prop}\label{prop:half vol2}
Let $M$ be an $n$-dimensional, 
connected complete Riemannian manifold with boundary with Riemannian metric $g$
such that $\ric_{M}\geq (n-1)\kappa$ and $H_{\bm} \geq \lambda$.
Suppose $\partial M$ is compact.
Assume that
there exists $R\in (0,\bar{C}_{\kappa,\lambda}]\setminus \{\infty\}$ such that
for all $r\in (0,R]$,
we have
\begin{equation*}
\frac{\vol_{g} B_{R}(\bm)}{\vol_{g} B_{r}(\bm)}= \frac{\vol_{g^{n}_{\kappa,\lambda}} B_{R}(\partial M^{n}_{\kappa,\lambda})}{\vol_{g^{n}_{\kappa,\lambda}} B_{r}(\partial M^{n}_{\kappa,\lambda})}.
\end{equation*}
Then for every $r\in (0,R)$,
the metric space $(B_{r}(\bm),d_{B_{r}(\bm)})$ is isometric to $([0,r]\times_{\kappa,\lambda}\bm,d_{\kappa,\lambda,r})$.
\end{prop}
\begin{proof}
Take $r\in (0,R)$.
By Lemma \ref{lem:half rigid},
for all $x\in \bm$,
we have $\tau(x)>r$;
in particular,
$B_{r}(\bm)\cap \cut\bm=\emptyset$.
Each connected component of $\bm$ one-to-one corresponds to the connected component of $B_{r}(\bm)$.
Therefore,
we may assume that $B_{r}(\bm)$ is connected.

By Lemma \ref{lem:comp2},
for all $t\in (0,R]$ and $x\in \bm$, 
we have $\thetae(t,x)=s^{n-1}_{\kappa,\lambda}(t)$.
Choose an orthonormal basis $\{ e_{x,i} \}_{i=1}^{n-1}$ of $T_{x}\bm$.
For each $i=1,\dots,n-1$,
let $Y_{x,i}$ be the $\bm$-Jacobi field along $\gamma_{x}$
with initial conditions $Y_{x,i}(0)=e_{x,i}$ and $Y_{x,i}'(0)=-A_{u_{x}}e_{x,i}$.
For all $t\in [0,\min \{R,\bar{C}_{\kappa,\lambda}\}]$ and $x\in \bm$, 
we have $Y_{x,i}(t)=s_{\kappa,\lambda}(t)\, E_{x,i}(t)$,
where $E_{x,i}$ are the parallel vector fields along $\gamma_{x}$ with initial condition $E_{x,i}(0)=e_{x,i}$ (see Remark \ref{rem:volume eq}).
Define a map $\Phi:[0,r]\times \bm\to B_{r}(\bm)$ by $\Phi(t,x):=\gamma_{x}(t)$.
For every $p\in (0,r)\times \bm$,
the map $D(\Phi|_{(0,r)\times \bm})_{p}$ sends an orthonormal basis of $T_{p}([0,r]\times \bm)$ to that of $T_{\Phi(p)}B_{r}(\bm)$,
and for every $x\in \{0,r\}\times \bm$,
the map $D(\Phi|_{\{0,r\}\times \bm})_{x}$ sends an orthonormal basis of $T_{x}(\{0,r\}\times \bm)$ to that of $T_{\Phi(x)}\partial (B_{r}(\bm))$.
Hence,
$\Phi$ is a Riemannian isometry with boundary from $[0,r]\times_{\kappa,\lambda} \bm$ to $B_{r}(\bm)$.
\end{proof}
\subsection{Proof of Theorem \ref{thm:volume growth distance rigidity}}
Let $M$ be an $n$-dimensional,
connected complete Riemannian manifold with boundary with Riemannian metric $g$
such that $\ric_{M}\geq (n-1)\kappa$ and $H_{\bm} \geq \lambda$.
Suppose $\partial M$ is compact.
We assume
\begin{equation*}
\liminf_{r\to \infty}\frac{\vol_{g}B_{r}(\bm)}{f_{n,\kappa,\lambda}(r)}\geq \vol_{h}\bm.
\end{equation*}
By Theorem \ref{thm:volume comparison} and Proposition \ref{prop:volume growth sup},
for all $r,R\in (0,\infty)$ with $r\leq R$,
\begin{equation*}
\frac{\vol_{g}B_{R}(\bm)}{f_{n,\kappa,\lambda}(R)}=\frac{\vol_{g}B_{r}(\bm)}{f_{n,\kappa,\lambda}(r)}=\vol_{h}\bm.
\end{equation*}
If $\kappa$ and $\lambda$ satisfy the ball-condition,
then for all $r\in (0,\const]$
we have
\begin{equation*}
\frac{\vol_{g} B_{\const}(\bm)}{\vol_{g} B_{r}(\bm)}= \frac{\vol_{g^{n}_{\kappa,\lambda}} B_{\const}(\partial M^{n}_{\kappa,\lambda})}{\vol_{g^{n}_{\kappa,\lambda}} B_{r}(\partial M^{n}_{\kappa,\lambda})};
\end{equation*}
in particular,
Lemmas \ref{lem:shinitai2} and \ref{lem:half rigid} imply that
$\tau$ is equal to $\const$ on $\bm$.
If $\kappa$ and $\lambda$ do not satisfy the ball-condition,
then for all $R\in (0,\infty)$ and $r\in (0,R]$
we have
\begin{equation*}
\frac{\vol_{g} B_{R}(\bm)}{\vol_{g} B_{r}(\bm)}= \frac{\vol_{g^{n}_{\kappa,\lambda}} B_{R}(\partial M^{n}_{\kappa,\lambda})}{\vol_{g^{n}_{\kappa,\lambda}} B_{r}(\partial M^{n}_{\kappa,\lambda})};
\end{equation*}
in particular,
Lemma \ref{lem:half rigid} implies that
for all $x\in \bm$,
we have $\tau(x)=\infty$.
It follows that
$\tau$ coincides with $\bconst$ on $\bm$.

If $\kappa$ and $\lambda$ satisfy the ball-condition,
then Lemmas \ref{lem:compact} and \ref{lem:shinitai2} imply that $M$ is compact;
in particular,
there exists a point $p\in M$ such that $\rho_{\bm}(p)=\dm=\const$.
Hence,
from Theorem \ref{thm:ball rigid},
it follows that $(M,d_{M})$ is isometric to $(\ball,d_{\ball})$.

If $\kappa$ and $\lambda$ do not satisfy the ball-condition,
then $\cut \bm=\emptyset$.
From Lemma \ref{lem:shinitai4},
it follows that $\bm$ is connected.
Take a sequence $\{r_{i}\}$ with $r_{i}\to \infty$.
By Proposition \ref{prop:half vol2},
for each $r_{i}$,
we obtain a Riemannian isometry $\Phi_{i}:[0,r_{i}]\times_{\kappa,\lambda}\bm\to B_{r_{i}}(\bm)$ with boundary
from $[0,r_{i}]\times_{\kappa,\lambda}\bm$ to $B_{r_{i}}(\bm)$ defined by $\Phi_{i}(t,x):=\gamma_{x}(t)$.
Since for all $x\in\bm$
it holds that $\tau(x)=\infty$,
there exists a Riemannian isometry $\Phi:[0,\infty)\times_{\kappa,\lambda}\bm\to M$ with boundary
from $[0,\infty)\times_{\kappa,\lambda}\bm$ to $M$
defined by $\Phi(t,x):=\gamma_{x}(t)$ satisfying $\Phi|_{[0,r_{i}]\times_{\kappa,\lambda}\bm}=\Phi_{i}$.
Hence,
$(M,d_{M})$ is isometric to $([0,\infty)\times_{\kappa,\lambda}\bm,d_{\kappa,\lambda})$.
We complete the proof.
\qed
\subsection{Curvature of the boundary}
It seems that
the following is well-known,
especially in a submanifold setting (see e.g., Proposition 9.36 in \cite{B}).
For the sake of the readers,
we give a proof in our setting.
\begin{lem}\label{lem:boundary Ricci formula}
Let $M$ be an $n$-dimensional Riemannian manifold with boundary with Riemannian metric $g$.
Let $h$ denote the induced Riemannian metric on $\bm$.
Take a point $x\in \bm$, 
and choose an orthonormal basis $\{ e_{x,i} \}_{i=1}^{n-1}$ of $T_{x}\bm$.
Put $u:=e_{x,1}$.
Then
\begin{equation*}
\ric_{h}(u)=\ric_{g}(u)-K_{g}(u_{x},u)+\tr A_{S(u,u)}-\sum_{i=1}^{n-1} \Vert S(u,e_{x,i})\Vert^{2},
\end{equation*}
where $K_{g}(u_{x},u)$ is the sectional curvature at $x$ in $(M,g)$ determined by $u_{x}$ and $u$.
\end{lem}
\begin{proof}
Note that $\ric_{h}(u)=\sum_{i=2}^{n-1} K_{h}(u,e_{x,i})$.
By the Gauss formula, 
\begin{equation*}
\ric_{h}(u)=\sum_{i=2}^{n-1} \left(K_{g}(u,e_{x,i})+g(S(u,u),S(e_{x,i},e_{x,i}))-\Vert S(u,e_{x,i})\Vert^{2}\right).
\end{equation*}
Since $u,e_{x,2},\dots,e_{x,n-1},u_{x}$ are orthogonal to each other,
we have
\begin{equation*}
\ric_{g}(u)=\sum_{i=2}^{n-1} K_{g}(u,e_{x,i})+K_{g}(u,u_{x}).
\end{equation*}
On the other hand,
we see
\begin{equation*}
\sum_{i=1}^{n-1}g(S(u,u),S(e_{x,i},e_{x,i}))=\sum_{i=1}^{n-1}g(A_{S(u,u)}e_{x,i},e_{x,i})=\tr A_{S(u,u)}.
\end{equation*}
Combining these equalities, 
we have the formula.
\end{proof}
To study our rigidity cases,
we need the following:
\begin{lem}\label{lem:rigid boundary Ricci}
Let $M$ be an $n$-dimensional,
connected complete Riemannian manifold with boundary with Riemannian metric $g$
such that $\ric_{M}\geq (n-1)\kappa$.
If $(M,d_{M})$ is isometric to $([0,\infty)\times_{\kappa,\lambda} \bm,d_{\kappa,\lambda})$,
then we have $\ric_{\bm}\geq (n-2)(\kappa+\lambda^{2})$.
\end{lem}
\begin{proof}
There exists a Riemannian isometry with boundary from $M$ to $[0,\infty)\times_{\kappa,\lambda}\bm$.
For each $x\in \bm$,
choose an orthonormal basis $\{e_{x,i}\}_{i=1}^{n-1}$ of $T_{x}\bm$.
For each $i=1,\dots,n-1$,
let $Y_{x,i}$ be the $\bm$-Jacobi field along $\gamma_{x}$
with initial conditions $Y_{x,i}(0)=e_{x,i}$ and $Y'_{x,i}(0)=-A_{u_{x}}e_{x,i}$.
We have $Y_{x,i}(t)=s_{\kappa,\lambda}(t)E_{x,i}(t)$,
where $E_{x,i}$ are the parallel vector fields along $\gamma_{x}$ with initial condition $E_{x,i}(0)=e_{x,i}$.
Then $A_{u_{x}}e_{x,i}=-Y'_{x,i}(0)=\lambda e_{x,i}$ and $Y''_{x,1}(0)=\kappa e_{x,1}$.
Hence, $\tr A_{u_{x}}=(n-1)\lambda$ and $K_{g}(u_{x},e_{x,1})=\kappa$.
For all $i$
we have $S(e_{x,i},e_{x,i})=\lambda u_{x}$,
and for all $i\neq j$
we have $S(e_{x,i},e_{x,j})=0_{x}$.
By Lemma \ref{lem:boundary Ricci formula} and $\ric_{M}\geq (n-1)\kappa$,
we have $\ric_{\bm}\geq (n-2)(\kappa+\lambda^{2})$.
\end{proof}
\subsection{Complement rigidity}
For $\kappa>0$,
let $M$ be an $n$-dimensional,
connected complete Riemmanian manifold (without boundary) with Riemannian metric $g$ such that $\ric_{M}\geq (n-1)\kappa$.
By the Bishop volume comparison theorem (\cite{BC}),
$\vol_{g} M\leq \vol M^{n}_{\kappa}$;
the equality holds if and only if $M$ is isometric to $M^{n}_{\kappa}$.

The following is concerned with the complements of metric balls.
\begin{cor}\label{cor:complement}
Let $\kappa\in \mathbb{R}$ and $-\lambda\in \mathbb{R}$ satisfy the ball-condition.
Let $M$ be an $n$-dimensional,
connected complete Riemannian manifold with boundary with Riemannian metric $g$
such that $\ric_{M}\geq (n-1)\kappa$ and $H_{\bm} \geq \lambda$.
Suppose $\partial M$ is compact.
Let $h$ denote the induced Riemannian metric on $\bm$.
If
\begin{equation*}
\liminf_{r\to \infty} \frac{\vol_{g}B_{r}(\bm)}{f_{n,\kappa,\lambda}(r)}\geq \vol_{h}\bm,\quad \vol_{h}\bm\geq \vol_{h^{n-1}_{\kappa,-\lambda}} \partial B^{n}_{\kappa,-\lambda},
\end{equation*}
then $(M,d_{M})$ is isometric to $(M^{n}_{\kappa}\setminus \inte B^{n}_{\kappa,-\lambda},d_{M^{n}_{\kappa}\setminus \inte B^{n}_{\kappa,-\lambda}})$.
\end{cor}
\begin{proof}
By Theorem \ref{thm:volume growth distance rigidity},
$(M,d_{M})$ is isometric to $([0,\infty)\times_{\kappa,\lambda}\bm,d_{\kappa,\lambda})$.
Lemma \ref{lem:rigid boundary Ricci} implies $\ric_{\bm}\geq (n-2)(\kappa+\lambda^{2})$.
Since $\kappa$ and $-\lambda$ satisfy the ball-condition,
$(\bm,h)$ is a connected complete Riemannian manifold of positive Ricci curvature.
By the assumption $\vol_{h}\bm\geq \vol_{h^{n-1}_{\kappa,-\lambda}} \partial B^{n}_{\kappa,-\lambda}$,
and by the Bishop volume comparison theorem,
$(\bm,h)$ is isometric to $(\partial B^{n}_{\kappa,-\lambda},h^{n-1}_{\kappa,-\lambda})$.
It turns out that $M$ and $M^{n}_{\kappa}\setminus \inte B^{n}_{\kappa,-\lambda}$ are isometric to each other as metric spaces.
\end{proof}

\section{Splitting theorems}\label{sec:splitting}
Let $M$ be a connected complete Riemannian manifold with boundary.
A normal geodesic $\gamma:[0,\infty)\to M$ is said to be a \textit{ray}
if for all $s,t\in [0,\infty)$, 
we have $d_{M}(\gamma(s),\gamma(t))=|s-t|$.
For a ray $\gamma:[0,\infty)\to M$, 
the function $b_{\gamma}:M\to \mathbb{R}$ defined as
\begin{equation*}
b_{\gamma}(p):=\lim_{t\to \infty}(t-d_{M}(p,\gamma(t)))
\end{equation*}
is called the \textit{busemann function} of $\gamma$.

\begin{lem}\label{lem:geometric splitting}
Let $M$ be a connected complete Riemannian manifold with boundary.
Suppose that for some $x_{0}\in \bm$,
we have $\tau(x_{0})=\infty$.
Take a point $p\in \inte M$.
If $b_{\gamma_{x_{0}}}(p)=\rho_{\bm}(p)$,
then $p\notin \cut \bm$.
Moreover,
for the unique foot point $x$ on $\bm$ of $p$,
we have $\tau(x)=\infty$.
\end{lem}
\begin{proof}
Since $\tau(x_{0})=\infty$,
the normal geodesic $\gamma_{x_{0}}:[0,\infty)\to M$ is a ray.
Since $\rho_{\bm}$ is $1$-Lipschitz,
for all $q\in M$,
we have $b_{\gamma_{x_{0}}}(q)\leq \rho_{\bm}(q)$.

Take a foot point $x$ on $\bm$ of $p$.
Suppose $p\in \cut \bm$.
We have $\tau(x)<\infty$
and $p=\gamma_{x}(\tau(x))$.
Take $\epsilon>0$ with $B_{\epsilon}(p)\subset \inte M$,
and a sequence $\{t_{i}\}$ with $t_{i}\to \infty$.
For each $i$,
we take a normal minimal geodesic $\gamma_{i}:[0,l_{i}]\to M$ from $p$ to $\gamma_{x_{0}}(t_{i})$.
Then $\gamma_{i}|_{[0,\epsilon)}$ lies in $\inte M$.
Put $u_{i}:=\gamma'_{i}(0)\in U_{p}M$.
By taking a subsequence,
for some $u\in U_{p}M$,
we have $u_{i}\to u$ in $U_{p}M$.
We denote by $\gamma_{u}:[0,T)\to M$ the normal geodesic
with initial conditions $\gamma_{u}(0)=p$ and $\gamma'_{u}(0)=u$.
We have
\begin{equation*}
t_{i}-d_{M}(p,\gamma_{x_{0}}(t_{i}))=-\epsilon+(t_{i}-d_{M}(\gamma_{i}(\epsilon),\gamma_{x_{0}}(t_{i}))).
\end{equation*}
By letting $i\to \infty$,
we have $b_{\gamma_{x_{0}}}(p)=-\epsilon+b_{\gamma_{x_{0}}}(\gamma_{u}(\epsilon))$.
From the assumption $b_{\gamma_{x_{0}}}(p)=\rho_{\bm}(p)$,
it follows that $\rho_{\bm}(p)\leq -\epsilon+\rho_{\bm}(\gamma_{u}(\epsilon))$.
On the other hand,
since $\rho_{\bm}$ is $1$-Lipschitz,
we have the opposite.
Therefore,
$d_{M}(x,\gamma_{u}(\epsilon))$ is equal to $d_{M}(x,p)+d_{M}(p,\gamma_{u}(\epsilon))$;
in particular, 
we see $u=\gamma'_{x}(\tau(x))$.
Furthermore,
$\rho_{\bm}(\gamma_{x}(\tau(x)+\epsilon))=\tau(x)+\epsilon$.
This contradicts the definition of $\tau$.
Hence,
$p\notin \cut \bm$,
and $x$ is the unique foot point on $\bm$ of $p$.

Put $l:=\rho_{\bm}(p)$.
We see that for every sufficiently small $\epsilon>0$,
we have $b_{\gamma_{x_{0}}}(\gamma_{x}(l+\epsilon))=\rho_{\bm}(\gamma_{x}(l+\epsilon))$.
In particular,
for all $t\in [l,\infty)$,
we have $b_{\gamma_{x_{0}}}(\gamma_{x}(t))=\rho_{\bm}(\gamma_{x}(t))$.
It follows that $\tau(x)=\infty$.
\end{proof}
Let $M$ be a connected complete Riemannian manifold with boundary,
and let $\gamma:[0,\infty)\to M$ be a ray.
Take $p\in \inte M$,
and a sequence $\{t_{i}\}$ with $t_{i}\to \infty$.
For each $i$,
let $\gamma_{i}:[0,l_{i}]\to M$ be a normal minimal geodesic from $p$ to $\gamma(t_{i})$.
Since $\gamma$ is a ray,
we have $l_{i}\to \infty$.
Take a sequence $\{T_{j}\}$ with $T_{j}\to \infty$.
Since $M$ is proper,
there exists a subsequence $\{\gamma_{1,i}\}$ of $\{\gamma_{i}\}$,
and a normal minimal geodesic $\gamma_{p,1}:[0,T_{1}]\to M$ from $p$ to $\gamma_{p,1}(T_{1})$
such that $\gamma_{1,i}|_{[0,T_{1}]}$ uniformly converges to $\gamma_{p,1}$.
Furthermore,
there exists a subsequence $\{\gamma_{2,i}\}$ of $\{\gamma_{1,i}\}$,
and a normal minimal geodesic $\gamma_{p,2}:[0,T_{2}]\to M$ from $p$ to $\gamma_{p,2}(T_{2})$
such that $\gamma_{2,i}|_{[0,T_{2}]}$ uniformly converges to $\gamma_{p,2}$,
where $\gamma_{p,2}|_{[0,T_{1}]}=\gamma_{p,1}$.
By a diagonal argument,
we obtain a subsequence $\{\gamma_{k}\}$ of $\{\gamma_{i}\}$,
and a ray $\gamma_{p}:[0,\infty)\to M$ such that
for every $t\in (0,\infty)$,
we have $\gamma_{k}(t)\to \gamma_{p}(t)$ as $k\to \infty$.
We call such a ray $\gamma_{p}$ an \textit{asymptote for} $\gamma$ \textit{from} $p$.
\begin{lem}\label{lem:asymptote}
Let $M$ be a connected complete Riemannian manifold with boundary.
Suppose that for some $x_{0}\in \bm$,
we have $\tau(x_{0})=\infty$.
Take $l>0$,
and put $p:=\gamma_{x_{0}}(l)$.
Then there exists $\epsilon>0$ such that
for all $q\in B_{\epsilon}(p)$,
all asymptotes for the ray $\gamma_{x_{0}}$ from $q$ lie in $\inte M$.
\end{lem}
\begin{proof}
The proof is by contradiction.
Suppose that
there exists a sequence $\{q_{i}\}$ in $\inte M$ with $q_{i}\to p$ such that
for each $i$,
there exists an asymptote $\gamma_{i}$ for $\gamma_{x_{0}}$ from $q_{i}$ such that
$\gamma_{i}$ does not lie in $\inte M$.
Now,
$M$ is proper.
Therefore,
by taking a subsequence of $\{\gamma_{i}\}$,
we may assume that
there exists a ray $\gamma_{p}:[0,\infty)\to M$ such that
for every $t\in [0,\infty)$,
we have $\gamma_{i}(t)\to \gamma_{p}(t)$ as $i\to \infty$.

Fix $i$.
Since $\gamma_{i}$ is an asymptote for $\gamma_{x_{0}}$ from $q_{i}$,
there exists a sequence $\{t_{i_{k}}\}$ with $t_{i_{k}}\to \infty$ as $k\to \infty$,
and for every $k$
there exists a normal minimal geodesic $\gamma_{i_{k}}$ in $M$ from $q_{i}$ to $\gamma_{x_{0}}(t_{i_{k}})$ such that
for every $t\in (0,\infty)$
we have $\gamma_{i_{k}}(t)\to \gamma_{i}(t)$ as $k\to \infty$.
For a fixed $t\in (0,\infty)$,
and for every $k$,
we have
\begin{equation*}
t_{i_{k}}-d_{M}(q_{i},\gamma_{x_{0}}(t_{i_{k}}))=-t+\left(t_{i_{k}}-d_{M}(\gamma_{i_{k}}(t),\gamma_{x_{0}}(t_{i_{k}}))\right).
\end{equation*}
Letting $k\to \infty$,
we have $b_{\gamma_{x_{0}}}(q_{i})=-t+b_{\gamma_{x_{0}}}(\gamma_{i}(t))$.
By letting $i \to \infty$,
we obtain $b_{\gamma_{x_{0}}}(p)=-t+b_{\gamma_{x_{0}}}(\gamma_{p}(t))$.

Since $\rho_{\bm}$ is $1$-Lipschitz,
and since $\tau(x_{0})=\infty$,
we have $b_{\gamma_{x_{0}}}\leq \rho_{\bm}$ on $M$,
and the equality holds at $p$.
Furthermore,
for every $t \in (0,\infty)$
we have $b_{\gamma_{x_{0}}}(p)=-t+b_{\gamma_{x_{0}}}(\gamma_{p}(t))$.
Therefore,
for every $t\in(0,\infty)$,
\begin{align*}
d_{M}(\gamma_{p}(t),x_{0})&\geq \rho_{\bm}(\gamma_{p}(t)) \geq b_{\gamma_{x_{0}}}(\gamma_{p}(t))
                                                                                                    = t+\rho_{\bm}(p)\\
                                                                                                  &=d_{M}(\gamma_{p}(t),p)+d_{M}(p,x_{0}).
\end{align*}
From the triangle inequality,
it follows that
$d_{M}(\gamma_{p}(t),x_{0})$ is equal to $d_{M}(\gamma_{p}(t),p)+d_{M}(p,x_{0})$.
In particular,
$\gamma_{p}|_{[0,\infty)}$ coincides with $\gamma_{x_{0}}|_{[l,\infty)}$.
Since $q_{i}\in \inte M$ for each $i$,
we have $u_{i}:=\gamma'_{i}(0)\in U_{q_{i}}M$.
We have $q_{i}\to p$ in $M$.
Therefore,
by taking a subsequence of $\{u_{i}\}$,
for some $u\in U_{p}M$
we have $u_{i}\to u$ in the unit tangent bundle on $\inte M$.
Since $\gamma_{p}|_{[0,\infty)}$ coincides with $\gamma_{x_{0}}|_{[l,\infty)}$,
we have $u=\gamma'_{x_{0}}(l)$.
Put
\begin{equation*}
t_{i}:=\sup\{t>0\mid \gamma_{i}([0,t))\subset \inte M\}
\end{equation*}
and $x_{i}:=\gamma_{i}(t_{i})\in \bm$.
Since all $\gamma_{i}$ are asymptotes for $\gamma_{x_{0}}$,
and since $\rho_{\bm}(x_{i})=0$ for all $i$,
we have
\begin{equation*}
b_{\gamma_{x_{0}}}(q_{i})=-t_{i}+b_{\gamma_{x_{0}}}(x_{i})\leq -t_{i}.
\end{equation*}
We see $b_{\gamma_{x_{0}}}(q_{i})\to l$ as $i\to \infty$.
Therefore,
the sequence $\{t_{i}\}$ does not diverge.
We may assume that for some $x\in \bm$,
the sequence $\{x_{i}\}$ converges to $x$ in $\bm$.
Since $u=\gamma'_{x_{0}}(l)$,
the ray $\gamma_{x_{0}}$ passes through $x$.
This contradicts that $\gamma_{x_{0}}|_{(0,\infty)}$ lies in $\inte M$.
\end{proof}
Let $M$ be a connected complete Riemannian manifold with boundary.
Take a point $p\in \inte M$,
and a continuous function $f:M\to \mathbb{R}$.
We say that 
a function $\bar{f}:M\to \mathbb{R}$ is a \textit{support function of} $f$ \textit{at} $p$
if we have $\bar{f}(p)=f(p)$,
and for all $q\in M$,
we have $\bar{f}(q)\leq f(q)$.

Take a domain $U$ in $\inte M$.
We say that $f$ is \textit{subharmonic in a barrier sense on} $U$
if for each $\epsilon>0$,
and for each $p\in U$,
there exists a support function $f_{p,\epsilon}:M\to \mathbb{R}$ of $f$ at $p$
such that $f_{p,\epsilon}$ is smooth on an open neighborhood of $p$,
and $\Delta f_{p,\epsilon}(p)\leq \epsilon$.
The Calabi maximal principle in \cite{C} tells us that
if a function that is subharmonic in a barrier sense on $U$ takes the maximal value at a point in $U$,
then the function must be constant.

We prove Theorem \ref{thm:splitting} by using the Calabi maximal principle in \cite{C}.
\begin{proof}[Proof of Theorem \ref{thm:splitting}]
For $\kappa \leq 0$,
let $M$ be an $n$-dimensional,
connected complete Riemannian manifold with boundary
such that $\ric_{M}\geq (n-1)\kappa$ and $H_{\bm} \geq \sqrt{\vert \kappa \vert}$.
Assume that for $x\in \bm$,
we have $\tau(x)=\infty$.
Let $\bm_{0}$ be the connected component of $\bm$ containing $x$.
Put
\begin{equation*}
\Omega:=\{y\in \bm_{0} \mid \tau(y)=\infty \}.
\end{equation*}
The assumption implies $\Omega \neq \emptyset$.
By the continuity of the function $\tau$,
we see that $\Omega$ is closed in $\bm_{0}$.

We show the openness of $\Omega$ in $\bm_{0}$.
Let $x_{0}\in \Omega$.
Take $l>0$,
and put $p_{0}:=\gamma_{x_{0}}(l)$.
By Lemma \ref{lem:asymptote},
there exists a sufficiently small open neighborhood $U$ of $p_{0}$ in $\inte M$ with $U\subset D_{\bm}$ such that
for each $q\in U$,
the unique foot point on $\bm$ of $q$ belongs to $\bm_{0}$,
and all asymptotes for $\gamma_{x_{0}}$ from $q$ lie in $\inte M$.

We prove that
the function $b_{\gamma_{x_{0}}}-\rho_{\bm}$ is subharmonic in a barrier sense on $U$.
By Proposition \ref{prop:distance function},
$\rho_{\bm}$ is smooth on $U$.
Fix a point $q_{0}\in U$,
and take an asymptote $\gamma_{q_{0}}:[0,\infty)\to M$ for $\gamma_{x_{0}}$ from $q_{0}$.
For $t>0$,
define a function $b_{\gamma_{x_{0}},t}:M\to \mathbb{R}$ by
\begin{equation*}
b_{\gamma_{x_{0}},t}(p):=b_{\gamma_{x_{0}}}(q_{0})+t-d_{M}(p,\gamma_{q_{0}}(t)).
\end{equation*}
We see that
$b_{\gamma_{x_{0}},t}-\rho_{\bm}$ is a support function of $b_{\gamma_{x_{0}}}-\rho_{\bm}$ at $q_{0}$.
Since $\gamma_{q_{0}}$ is a ray contained in $\inte M$,
for every $t\in(0,\infty)$,
the function $b_{\gamma_{x_{0}},t}$ is smooth on a neighborhood of $q_{0}$ in $\inte M$.
By Lemma \ref{lem:point Laplacian},
we have $\Delta b_{\gamma_{x_{0}},t}(q_{0})\leq (n-1)(s'_{\kappa}(t)/s_{\kappa}(t))$.
Note that $s'_{\kappa}(t)/s_{\kappa}(t)\to \sqrt{\vert \kappa \vert}$ as $t\to \infty$.
On the other hand,
by Theorem \ref{thm:boundary Laplacian},
for all $q\in U$,
we have $\Delta \rho_{\bm}(q)\geq (n-1)\sqrt{\vert \kappa \vert}$.
Hence,
$b_{\gamma_{x_{0}}}-\rho_{\bm}$ is subharmonic in a barrier sense on $U$.
The function $b_{\gamma_{x_{0}}}-\rho_{\bm}$ takes the maximal value $0$ at $p_{0}$.
The Calabi maximal principle in \cite{C} implies that
$b_{\gamma_{x_{0}}}$ coincides with $\rho_{\bm}$ on $U$.
From Lemma \ref{lem:geometric splitting},
it follows that $\Omega$ is open in $\bm_{0}$.

For all $x\in \bm_{0}$,
we have $\tau(x)=\infty$.
We put
\begin{equation*}
TD_{\bm_{0}}:=\bigcup_{x\in \bm_{0}}\{t\,u_{x} \mid t\in (0,\infty)\}.
\end{equation*}
By Lemma \ref{lem:regular and injective},
$\expp|_{TD_{\bm_{0}}}:TD_{\bm_{0}}\to \expp(TD_{\bm_{0}})$ is a diffeomorphism.
The set $TD_{\bm_{0}}$ is open and closed in $TD_{\bm}\setminus 0(\tbp)$.
Therefore,
$\expp(TD_{\bm_{0}})$ is also open and closed in $\inte M$.
Since $\inte M$ is connected,
$\expp(TD_{\bm_{0}})$ coincides with $\inte M$;
in particular,
$\bm$ is connected
and $\cut \bm=\emptyset$.
Note that
$\rho_{\bm}$ is smooth on $\inte M$.

Take $p\in \inte M$
and the unique foot point $x_{p}$ on $\bm$ of $p$.
Since $\tau(x_{p})=\infty$,
the maximal principle argument implies that
$b_{\gamma_{x_{p}}}$ coincides with $\rho_{\bm}$ on a neighborhood $V$ of $p$ in $\inte M$;
in particular,
$b_{\gamma_{x_{p}}}$ is smooth on $V$,
and $\Delta \rho_{\bm}(p)=(n-1)\sqrt{\vert \kappa \vert}$.
It follows that
the equality in Theorem \ref{thm:boundary Laplacian} holds on $\inte M$.
For each $x\in \bm$,
choose an orthonormal basis $\{e_{x,i}\}_{i=1}^{n-1}$ of $T_{x}\bm$.
For each $i=1,\dots,n-1$,
let $Y_{x,i}$ be the $\bm$-Jacobi field along $\gamma_{x}$
with initial conditions $Y_{x,i}(0)=e_{x,i}$ and $Y'_{x,i}(0)=-A_{u_{x}}e_{x,i}$.
Then we have $Y_{x,i}(t)=s_{\kappa,\sqrt{\vert \kappa \vert}}(t)E_{x,i}(t)$,
where $E_{x,i}$ is the parallel vector fields along $\gamma_{x}$ with initial condition $E_{x,i}(0)=e_{x,i}$ (see Remark \ref{rem:lap eq}).
Define a map $\Phi:[0,\infty)\times \bm\to M$ by $\Phi(t,x):=\gamma_{x}(t)$.
For every $p\in (0,\infty)\times \bm$,
the map $D(\Phi|_{(0,\infty)\times \bm})_{p}$ sends an orthonormal basis of $T_{p}((0,\infty)\times \bm)$ to that of $T_{\Phi(p)}M$,
and for every $x\in \{0\}\times \bm$,
the map $D(\Phi|_{\{0\}\times \bm})_{x}$ sends an orthonormal basis of $T_{x}(\{0\}\times \bm)$ to that of $T_{\Phi(x)}\bm$.
Therefore,
$\Phi$ is a Riemannian isometry with boundary from $[0,\infty)\times_{\kappa,\sqrt{\vert \kappa \vert}} \bm$ to $M$.
We complete the proof of Theorem \ref{thm:splitting}.
\end{proof}
The Cheeger-Gromoll splitting theorem (\cite{CG}) states that
if $M$ is an $n$-dimensional, 
connected complete Riemmanian manifold of non-negative Ricci curvature,
and if $M$ contains a line,
then there exists an $(n-1)$-dimensional Riemannian manifold $N$ of non-negative Ricci curvature such that
$M$ is isometric to the standard product $\mathbb{R}\times N$.
\begin{cor}
For $\kappa\leq 0$,
let $M$ be an $n$-dimensional,
connected complete Riemannian manifold with boundary
such that $\ric_{M}\geq (n-1)\kappa$ and $H_{\bm} \geq \sqrt{\vert \kappa \vert}$.
Suppose that for some $x\in \bm$,
we have $\tau(x)=\infty$.
Then there exist $k\in \{0,\dots,n-1\}$, 
and an $(n-1-k)$-dimensional,
connected complete Riemannian manifold $N$ of non-negative Ricci curvature containing no line such that
$(\bm,d_{\bm})$ is isometric to the standard product metric space $(\mathbb{R}^{k}\times N,d_{\mathbb{R}^{k}\times N})$.
In particular,
$(M,d_{M})$ is isometric to $([0,\infty)\times_{\kappa,\sqrt{\vert \kappa \vert}} (\mathbb{R}^{k}\times N),d_{\kappa,\sqrt{\vert \kappa \vert}})$.
\end{cor}
\begin{proof}
From Theorem \ref{thm:splitting},
it follows that the metric space $(M,d_{M})$ is isometric to $([0,\infty)\times_{\kappa,\sqrt{\vert \kappa \vert}}\bm,d_{\kappa,\sqrt{\vert \kappa \vert}})$.
Lemma \ref{lem:rigid boundary Ricci} implies $\ric_{\bm}\geq 0$.
Applying the Cheeger-Gromoll splitting theorem to $\bm$ inductively,
we see that $(\bm,d_{\bm})$ is isometric to $(\mathbb{R}^{k}\times N,d_{\mathbb{R}^{k}\times N})$ for some $k$.
\end{proof}

\section{The first eigenvalues}\label{sec:segment inequality}
\subsection{Lower bounds}
Let $M$ be a connected complete Riemannian manifold with boundary with Riemannian metric $g$.
For a relatively compact domain $\Omega$ in $M$
such that $\partial \Omega$ is a smooth hypersurface in $M$,
we denote by $\vol_{\partial \Omega}$ the Riemannian volume measure on $\partial \Omega$ induced from the induced Riemannian metric on $\partial \Omega$.
For $\alpha\in (0,\infty)$,
the \textit{Dirichlet $\alpha$-isoperimetric constant} $ID_{\alpha}(M)$ of $M$ is defined as
\begin{equation*}
ID_{\alpha}(M):=\inf_{\Omega}\, \frac{\vol_{\partial \Omega}\, \partial \Omega}{\left(\vol_{g} \Omega \right)^{1/\alpha}},
\end{equation*}
where the infimum is taken over all relatively compact domains $\Omega$ in $M$
such that $\partial \Omega$ is a smooth hypersurface in $M$ and $\partial \Omega \cap \partial M=\emptyset$.
The \textit{Dirichlet $\alpha$-Sobolev constant} $SD_{\alpha}(M)$ of $M$ is defined as
\begin{equation*}
SD_{\alpha}(M):=\inf_{f \in W^{1,1}_{0}(M)}\, \frac{\int_{M}\,  \Vert \nabla f \Vert \,d\,\vol_{g}}{\left(\int_{M}\,  \vert f \vert^{\alpha} \,d\,\vol_{g} \right)^{1/\alpha}}.
\end{equation*}
For all $\alpha \in (0,\infty)$,
we have $ID_{\alpha}(M)=SD_{\alpha}(M)$.
This relationship between 
the isoperimetric constant and the Sobolev constant has been 
formally established by Federer and Fleming in \cite{FF} (see e.g., Theorem 4 in Chapter 4 in \cite{Ch}, Theorem 9.5 in \cite{Li}),
and later used by Cheeger in \cite{Che} for the estimate of the first Dirichlet eigenvalue of the Laplacian.

The following volume estimate has been proved by Kasue in Proposition 4.1 in \cite{K5}.
\begin{prop}[\cite{K5}]\label{prop:Kasue volume estimate}
Let $M$ be an $n$-dimensional,
connected complete Riemannian manifold with boundary
with Riemannian metric $g$
such that $\ric_{M}\geq (n-1)\kappa$ and $H_{\bm} \geq \lambda$.
Let $\Omega$ be a relatively compact domain in $M$
such that $\partial \Omega$ is a smooth hypersurface in $M$.
Then
\begin{equation*}
\vol_{g} \Omega \leq \vol_{\partial \Omega}\, \partial \Omega \,\sup_{t\in (\delta_{1}(\Omega),\delta_{2}(\Omega))}\, \frac{\int^{\delta_{2}(\Omega)}_{t}\,  s^{n-1}_{\kappa,\lambda}(s)\, ds}{s^{n-1}_{\kappa,\lambda}(t)},
\end{equation*}
where $\delta_{1}(\Omega):=\inf_{p\in \Omega}\, \rho_{\bm}(p)$ and $\delta_{2}(\Omega):=\sup_{p\in \Omega} \,\rho_{\bm}(p)$.
\end{prop}
The equality case in Proposition \ref{prop:Kasue volume estimate} has been also studied in \cite{K5}.

We prove Theorem \ref{thm:p-Laplacian1}.
\begin{proof}[Proof of Theorem \ref{thm:p-Laplacian1}]
Let $M$ be an $n$-dimensional,
connected complete Riemannian manifold with boundary
with Riemannian metric $g$
such that $\ric_{M}\geq (n-1)\kappa,\,H_{\bm} \geq \lambda$ and $\dm \leq D$.
Suppose $\partial M$ is compact.
Recall that the positive constant $C(n,\kappa,\lambda,D)$ is defined as
\begin{equation*}
C(n,\kappa,\lambda,D):=\sup_{t\in [0,D)}\, \frac{\int^{D}_{t}\,  s^{n-1}_{\kappa,\lambda}(s)\, ds}{s^{n-1}_{\kappa,\lambda}(t)}.
\end{equation*}
Let $\Omega$ be a relatively compact domain in $M$ such that $\partial \Omega$ is a smooth hypersurface in $M$ and $\partial \Omega \cap \partial M=\emptyset$.
By Proposition \ref{prop:Kasue volume estimate},
\begin{equation*}\label{eq:Kasue}
\vol_{g} \Omega \leq \vol_{\partial \Omega} \partial \Omega\,\sup_{t\in (0,D)}\, \frac{\int^{D}_{t}\,  s^{n-1}_{\kappa,\lambda}(s)\, ds}{s^{n-1}_{\kappa,\lambda}(t)}=C(n,\kappa,\lambda,D) \,\vol_{\partial \Omega} \partial \Omega.
\end{equation*}
From the relationship $ID_{1}(M)=SD_{1}(M)$,
it follows that $SD_{1}(M)\geq C(n,\kappa,\lambda,D)^{-1}$.
Therefore,
for all $\phi \in W^{1,1}_{0}(M)$,
we have the following Poincar\'e inequality:
\begin{equation*}\label{eq:Poincare1}
\int_{M}\, \vert \phi \vert \,d\vol_{g}\leq C(n,\kappa,\lambda,D) \int_{M}\,\Vert \nabla \phi \Vert \,d\vol_{g}.
\end{equation*}

For a fixed $p\in (1,\infty)$,
let $\psi$ be a non-zero function in $W^{1,p}_{0}(M)$.
Put $q:=p\,(1-p)^{-1}$.
In the Poincar\'e inequality,
by replacing $\phi$ with $\vert \psi \vert^{p}$, 
and by the H\"older inequality,
we see
\begin{multline*}
\int_{M}\, \vert \psi \vert^{p} \,d\vol_{g} \leq p\,C(n,\kappa,\lambda,D)\, \int_{M}\, \vert \psi \vert^{p-1} \,  \Vert \nabla \psi\Vert \,d\vol_{g}\\
                                \leq p\, C(n,\kappa,\lambda,D)\, \left(\int_{M}\, \vert \psi \vert^{p}\,d \vol_{g}\right)^{1/q} \left(\int_{M}\, \Vert \nabla \psi\Vert^{p}\,d \vol_{g}\right)^{1/p}.\\
\end{multline*}
Considering the Rayleigh quotient $R_{p}(\psi)$,
we obtain the inequality $\mu_{1,p}(M)\geq (p\,C(n,\kappa,\lambda,D))^{-p}$.
This proves Theorem \ref{thm:p-Laplacian1}.
\end{proof}
We next prove Theorem \ref{thm:eigenvalue rigidity}.
\begin{proof}[Proof of Theorem \ref{thm:eigenvalue rigidity}]
Let $\kappa<0$ and $\lambda:=\sqrt{\vert \kappa \vert}$.
Let $M$ be an $n$-dimensional,
connected complete Riemannian manifold with boundary
such that $\ric_{M}\geq (n-1)\kappa$ and $H_{\bm} \geq \lambda$.
Suppose $\partial M$ is compact.
We put $D:=D(M,\bm)\in (0,\infty]$.
We have 
\begin{equation*}
C(n,\kappa,\lambda,D)=\left((n-1)\lambda \right)^{-1}\,\left(1-e^{-(n-1)\lambda\, D} \right).
\end{equation*}
The right hand side is monotone increasing as $D\to \infty$.
By Theorem \ref{thm:p-Laplacian1},
for all $p\in (1,\infty)$
we have $\mu_{1,p}(M)\geq ((n-1)\lambda/p)^{p}$.

We assume 
$\mu_{1,p}(M)=((n-1)\lambda/p)^{p}$.
By Theorem \ref{thm:p-Laplacian1},
we have $D=\infty$.
Therefore,
the compactness of $\bm$ and Lemma \ref{lem:compact} imply that
$M$ is noncompact.
It has been proved in Theorem C in \cite{K3} as a splitting theorem (see Subsection \ref{sec:Splitting theorems}) that
if $M$ is noncompact and $\bm$ is compact,
then $(M,d_{M})$ is isometric to $([0,\infty)\times_{\kappa,\lambda}\bm, d_{\kappa,\lambda})$.
Therefore,
$(M,d_{M})$ is isometric to $([0,\infty)\times_{\kappa,\lambda}\bm, d_{\kappa,\lambda})$.

Let $p=2$,
and let $(M,d_{M})$ be isometric to $([0,\infty)\times_{\kappa,\lambda}\bm, d_{\kappa,\lambda})$.
Let $\phi_{n,\kappa,\lambda}:[0,\infty)\to [0,\infty)$ be a smooth function defined by 
\begin{equation*}
\phi_{n,\kappa,\lambda}(t):=t\,e^{     \frac{(n-1)\lambda t}{2}       }.
\end{equation*}
Then the smooth function $\phi_{n,\kappa,\lambda}\circ \rho_{\bm}$ on $M$ satisfies
\begin{equation*}
\Delta_{2} (\phi_{n,\kappa,\lambda} \circ \rho_{\bm})=\left(\frac{(n-1)\lambda}{2}\right)^{2}\, (\phi_{n,\kappa,\lambda} \circ \rho_{\bm})
\end{equation*}
on $M$;
in particular,
\begin{equation*}
\mu_{1,2}(M)\leq R_{2}(\phi_{n,\kappa,\lambda} \circ \rho_{\bm})= \left(\frac{(n-1)\lambda}{2}\right)^{2}.
\end{equation*}
Therefore,
$\mu_{1,2}(M)=((n-1)\lambda/2)^{2}$.
This proves Theorem \ref{thm:eigenvalue rigidity}.
\end{proof}

\subsection{Segment inequality}
For $n\geq 2$,
$\kappa,\lambda\in \mathbb{R}$,
and $D\in (0,\bar{C}_{\kappa,\lambda}]$,
let $C_{1}(n,\kappa,\lambda,D)$ be the positive constant defined as
\begin{equation*}
C_{1}(n,\kappa,\lambda,D):=\sup_{l\in (0,D)}\, \sup_{t\in (0,l)}\, \frac{s^{n-1}_{\kappa,\lambda}(l)}{s^{n-1}_{\kappa,\lambda}(t)}.
\end{equation*}

We prove the following segment inequality:
\begin{prop}\label{prop:segment}
For $D\in (0,\bar{C}_{\kappa,\lambda}]\setminus \{\infty\}$,
let $M$ be an $n$-dimensional,
connected complete Riemannian manifold with boundary with Riemannian metric $g$
such that $\ric_{M}\geq (n-1)\kappa,\,H_{\bm} \geq \lambda$ and $\dm \leq D$.
Let $f:M\to \mathbb{R}$ be a non-negative integrable function on $M$,
and define a function $E_{f}:M\to \mathbb{R}$ by
\begin{equation*}
E_{f}(p):=\inf_{x\in \bm} \int^{\rho_{\bm}(p)}_{0}\, f(\gamma_{x}(t))\, dt,
\end{equation*}
where the infimum is taken over all foot points $x$ on $\bm$ of $p$.
Then
\begin{equation*}
\int_{M}\, E_{f}\,d\vol_{g}\leq C_{1}(n,\kappa,\lambda,D)D \int_{M}\,f\,d\vol_{g}.
\end{equation*}
\end{prop}
\begin{proof}
Put $C_{1}:=C_{1}(n,\kappa,\lambda,D)$.
Fix $x\in \bm$ and $l\in (0,\tau(x))$.
Observe that $x$ is the unique foot point on $\bm$ of $\gamma_{x}(l)$,
and $\gamma_{x}|_{(0,l]}$ lies in $\inte M$.
By Lemma \ref{lem:comp2},
for all $t\in [0,l]$ we have
\begin{equation*}
E_{f}(\gamma_{x}(l)) \thetae(l,x)\leq C_{1} \int^{l}_{0}\, f(\gamma_{x}(t))\thetae(t,x)\,dt.
\end{equation*}
Integrating the both sides,
we see
\begin{equation*}
\int^{\tau(x)}_{0} E_{f}(\gamma_{x}(l)) \thetae(l,x)\,dl \leq C_{1} D \int^{\tau(x)}_{0}\, f(\gamma_{x}(t))\thetae(t,x)\,dt.
\end{equation*}
Lemma \ref{lem:cut and regular} implies $M=\expp(\bigcup_{x\in \bm} \{tu_{x}\mid t\in [0,\tau(x)] \})$.
From Lemma \ref{lem:regular and injective},
it follows that $\expp|_{TD_{\bm}\setminus 0(T^{\perp}\bm)}$ is a diffeomorphism onto $D_{\bm}\setminus\bm$.
By Proposition \ref{prop:vol of cut},
we have $\vol_{g}\cut \bm=0$.
Integrating the both sides of the above inequality over $\bm$ with respect to $x$,
we obtain the desired segment inequality.
\end{proof}
From Proposition \ref{prop:segment},
we derive the following Poincar\'e inequality:
\begin{lem}\label{lem:Poincare}
For $D\in (0,\bar{C}_{\kappa,\lambda}]\setminus \{\infty\}$,
let $M$ be an $n$-dimensional,
connected complete Riemannian manifold with boundary with Riemannian metric $g$
such that $\ric_{M}\geq (n-1)\kappa, \, H_{\bm} \geq \lambda$ and $\dm \leq D$.
Let $\psi:M\to \mathbb{R}$ be a smooth integrable function on $M$ with $\psi|_{\bm}=0$.
Assume $\int_{M} \Vert \nabla \psi \Vert\, d\vol_{g}<\infty$.
Then
\begin{equation*}
\int_{M}\, \vert \psi \vert \,d\vol_{g}\leq C_{1}(n,\kappa,\lambda,D)D \int_{M}\,\Vert \nabla \psi\Vert \,d\vol_{g}.
\end{equation*}
\end{lem}
\begin{proof}
Put $f:=\Vert \nabla \psi \Vert$,
and let $E_{f}$ be the function defined in Proposition \ref{prop:segment}.
For each $p\in D_{\bm}$,
let $x$ be the foot point on $\bm$ of $p$.
By the Cauchy-Schwarz inequality,
we have
\begin{equation*}
\vert \psi(p)-\psi(x) \vert \leq \int^{\rho_{\bm}(p)}_{0}\,\left\vert g(\nabla \psi,\gamma'_{x}(t))\right\vert\,dt \leq E_{f}(p).
\end{equation*}
Since $\psi|_{\bm}=0$,
we have $\vert \psi(p)\vert\leq E_{f}(p)$.
Integrate the both sides of the inequality over $D_{\bm}$ with respect to $p$.
By Proposition \ref{prop:segment} and $\vol_{g} \cut \bm=0$,
we arrived at the desired inequality.
\end{proof}
As one of the applications of our segment inequality in Proposition \ref{prop:segment},
we show the following:
\begin{prop}\label{prop:p-Laplacian2}
For $D\in (0,\bar{C}_{\kappa,\lambda}]$,
let $M$ be an $n$-dimensional,
connected complete Riemannian manifold with boundary
such that $\ric_{M}\geq (n-1)\kappa, H_{\bm} \geq \lambda$ and $\dm \leq D$.
Let $M$ be compact.
Then for all $p\in (1,\infty)$,
we have
\begin{equation*}
\mu_{1,p}(M)\geq (\,p\,C_{1}(n,\kappa,\lambda,D)\,D\,)^{-p}.
\end{equation*}
\end{prop}
\begin{proof}
For a fixed $p\in (1,\infty)$,
let $\psi$ be a non-zero function in $W^{1,p}_{0}(M)$.
We may assume that $\psi$ is smooth on $M$.
In Lemma \ref{lem:Poincare},
by replacing $\psi$ with $\vert\psi\vert^{p}$,
we have
\begin{equation*}
\int_{M}\, \vert \psi \vert^{p} \,d\vol_{g} \leq p\,C_{1}(n,\kappa,\lambda,D)\, D\, \int_{M}\, \vert \psi \vert^{p-1} \,  \Vert \nabla \psi\Vert \,d\vol_{g}.
\end{equation*}
From the H\"older inequality,
we derive $R_{p}(\psi)\geq (\,p\,C_{1}(n,\kappa,\lambda,D)\,D\,)^{-p}$.
This proves Proposition \ref{prop:p-Laplacian2}.
\end{proof}
\begin{rem}
Proposition \ref{prop:p-Laplacian2} is weaker than Theorem \ref{thm:p-Laplacian1}.
We can prove that
the lower bound $(p\,C_{1}(n,\kappa,\lambda,D)\,D)^{-p}$ for $\mu_{1,p}$ in Proposition \ref{prop:p-Laplacian2}
is at most the lower bound $(p\,C(n,\kappa,\lambda,D))^{-p}$ in Theorem \ref{thm:p-Laplacian1}.
\end{rem}

\section{Measure contraction property}\label{sec:Measure contraction property}
Let $M$ be a connected complete Riemannian manifold with boundary with Riemannian metric $g$.
\subsection{Measure contraction inequalities}
Let $t \in (0,1)$.
For a point $p\in M$,
we say that
$q\in M$ is a $t$\textit{-extension point from} $\bm$ \textit{of} $p$
if $q$ satisfies the following:
(1) $\rho_{\bm}(p)/\rho_{\bm}(q)=t$;
(2) there exists a foot point $x$ on $\bm$ of $p$ with $q=\gamma_{x}(\rho_{\bm}(q))$.
We denote by $W_{t}$ the set of all points $p \in M$ for which
there exists a $t$-extension point from $\bm$ of $p$.

We first show the following:
\begin{lem}\label{lem:contraction point}
For every $t\in (0,1)$,
and for every $p\in W_{t}$,
there exists a unique foot point on $\bm$ of $p$.
In particular,
every $p\in W_{t}$ has a unique $t$-extension point from $\bm$.
\end{lem}
\begin{proof}
Take $p\in W_{t}$.
Let $q$ be a $t$-extension point from $\bm$ of $p$.
There exists a foot point $x$ on $\bm$ of $p$ such that $q=\gamma_{x}(\rho_{\bm}(q))$.
The definition of $\tau$ implies $\rho_{\bm}(q)\leq \tau(x)$.
Since $\rho_{\bm}(p)=t\rho_{\bm}(q)$,
it follows that $\rho_{\bm}(p)<\tau(x)$.
From Lemma \ref{lem:foot vector},
we derive $p=\gamma_{x}(\rho_{\bm}(p))$.
Lemma \ref{lem:cut} tells us that
$x$ is a unique foot point on $\bm$ of $p$.

Suppose that
there exist distinct $t$-extension points $q_{1},q_{2}\in M$ from $\bm$ of $p$.
By the definition,
it holds that $\rho_{\bm}(q_{1})=\rho_{\bm}(q_{2})$.
Furthermore,
for each $i=1,2$,
there exists a foot point $x_{i}$ on $\bm$ of $p$ with $q_{i}=\gamma_{x_{i}}(\rho_{\bm}(q_{i}))$.
Since $q_{1}\neq q_{2}$,
we have $x_{1}\neq x_{2}$.
This contradicts the property that
$p$ has a unique foot point on $\bm$.
\end{proof}
By Lemma \ref{lem:contraction point},
for every $t\in (0,1)$,
we can define a map $\Phi_{t}:W_{t} \to M$ by $\Phi_{t}(p):=q$,
where $q$ is a unique $t$-extension point from $\bm$ of $p$.
We call $\Phi_{t}$ the $t$\textit{-extension map from} $\bm$.
Notice that
for every $t\in (0,1)$,
the $t$-extension map $\Phi_{t}$ from $\bm$ is surjective and continuous.

Let $\Omega$ be a subset of $M$.
We say that
$x\in \bm$ is a \textit{foot point on $\bm$ of $\Omega$}
if there exists a point $p\in \Omega$ such that
$x$ is a foot point on $\bm$ of $p$.
We denote by $\Pi(\Omega)$ the set of all foot points on $\bm$ of $\Omega$.

We have the following property of the $t$-extension map $\Phi_{t}$ from $\bm$:
\begin{lem}\label{lem:contraction and foot point}
For $t\in (0,1)$,
let $\Phi_{t}$ be the $t$-extension map from $\bm$.
Let $\Omega$ be a subset of $M$.
Then $\Pi(\Phi^{-1}_{t}(\Omega))=\Pi(\Omega)$.
\end{lem}
\begin{proof}
First,
we show $\Pi(\Omega) \subset \Pi(\Phi^{-1}_{t}(\Omega))$.
Take $x\in \Pi(\Omega)$.
There exists $p\in \Omega$ such that
$x$ is a foot point on $\bm$ of $p$.
Put $p_{t}:=\gamma_{x}(t \rho_{\bm}(p))$.
It suffices to show that
$x$ is a foot point on $\bm$ of $p_{t}$,
and $p_{t}$ belongs to $\Phi^{-1}_{t}(\Omega)$.
Lemma \ref{lem:foot vector} implies $p=\gamma_{x}(\rho_{\bm}(p))$.
By the definition of $\tau$,
we see $\rho_{\bm}(p)\leq \tau(x)$;
in particular,
$t\rho_{\bm}(p)$ is smaller than $\tau(x)$.
From Lemma \ref{lem:cut},
it follows that
$x$ is a unique foot point on $\bm$ of $p_{t}$.
Furthermore,
we have $\rho_{\bm}(p_{t})=t \rho_{\bm}(p)$.
Hence,
$p$ is a $t$-extension point from $\bm$ of $p_{t}$.
By Lemma \ref{lem:contraction point},
$p$ is a unique $t$-extension point from $\bm$.
Since $p=\Phi_{t}(p_{t})$ and $p\in \Omega$,
we see $p_{t} \in \Phi^{-1}_{t}(\Omega)$.
This implies $x\in \Pi(\Phi^{-1}_{t}(\Omega))$.

Next,
we show the opposite.
Take $x\in \Pi(\Phi^{-1}_{t}(\Omega))$.
There exists $p\in \Phi^{-1}_{t}(\Omega)$ such that
$x$ is a foot point on $\bm$ of $p$.
By Lemma \ref{lem:contraction point},
$x$ is a unique foot point on $\bm$ of $p$.
By the definition of the $t$-extension point from $\bm$,
we see $\Phi_{t}(p)=\gamma_{x}(\rho_{\bm}(\Phi_{t}(p)))$.
Thus,
we have $\rho_{\bm}(\Phi_{t}(p))\leq \tau(x)$.
Hence,
$x$ is a foot point on $\bm$ of $\Phi_{t}(p)$.
Since $\Phi_{t}(p)\in \Omega$,
we have $x \in \Pi(\Omega)$.
This proves the lemma.
\end{proof}
For $t\in (0,1)$,
let $\Phi_{t}$ be the $t$-extension map from $\bm$.
Let $\Omega$ be a subset of $M$.
For $x \in \Pi(\Omega)$,
we put
\begin{equation*}
I_{\Omega,t,x}:=\{\, s \in (0,t\tau(x))\mid \gamma_{x}(s)\in \Phi^{-1}_{t}(\Omega) \,\}.
\end{equation*}

We prove the following:
\begin{lem}\label{lem:contraction and domain}
For $t\in (0,1)$,
let $\Phi_{t}$ be the $t$-extension map from $\bm$.
Suppose that
a subset $\Omega$ of $M$ is measurable,
and satisfies $\vol_{g}\Phi^{-1}_{t}(\Omega)<\infty$.
Then we have
\begin{equation*}
\vol_{g}\Phi^{-1}_{t}(\Omega)=\int_{\Pi(\Omega)}\int_{I_{\Omega,t,x}}\,\theta(s,x)\,ds\,d\vol_{h}.
\end{equation*}
\end{lem}
\begin{proof}
We put
\begin{align*}
A&:=\{\, \gamma_{x}(t \tau(x))\in \Phi^{-1}_{t}(\Omega) \mid x \in \Pi(\Omega), \tau(x)<\infty\},\\
B&:=\{\, \gamma_{x}(s) \mid x \in \Pi(\Omega), s\in I_{\Omega,t,x}\}.
\end{align*}
Note that
$A$ and $B$ are disjoint.

We show $\Phi^{-1}_{t}(\Omega)\setminus \bm=A\sqcup B$.
The definition of $I_{\Omega,t,x}$ implies $A\sqcup B \subset \Phi^{-1}_{t}(\Omega)\setminus \bm$.
To show the opposite,
take $p\in \Phi^{-1}_{t}(\Omega)\setminus \bm$,
and take a foot point $x$ on $\bm$ of $p$.
By Lemma \ref{lem:foot vector},
we see $p=\gamma_{x}(\rho_{\bm}(p))$.
From Lemma \ref{lem:contraction and foot point},
we derive $x\in \Pi(\Omega)$.
Now,
$p$ belongs to $W_{t}$.
Hence,
by Lemma \ref{lem:contraction point},
$x$ is a unique foot point on $\bm$ of $p$,
and there exists a unique $t$-extension point $q\in M$ from $\bm$ of $p$.
The $t$-extension point $q$ from $\bm$ of $p$ satisfies $t \rho_{\bm}(q)=\rho_{\bm}(p)$ and $q=\gamma_{x}(\rho_{\bm}(q))$.
The definition of $\tau$ implies $\rho_{\bm}(q)\leq \tau(x)$.
It holds that $\rho_{\bm}(p)\leq t \tau(x)$.
Since $x\in \Pi(\Omega)$ and $\rho_{\bm}(p)\in (0,t \tau(x)]$,
it follows that $\Phi^{-1}_{t}(\Omega)\setminus \bm \subset A\sqcup B$.

We next show that
$A$ is a null set of $M$.
We put
\begin{equation*}
\bar{A}:=\bigcup_{x\in \Pi(\Omega)} \{t \tau (x)u_{x}\mid \tau(x)<\infty \}.
\end{equation*}
Note that $A=\expp(\bar{A})$.
By Lemma \ref{lem:conti},
and by the Fubini theorem,
the graph $\{(x,t \tau(x)) \mid x\in \bm,\,\tau(x)<\infty \}$ of $t \tau$ is a null set of $\bm\times [0,\infty)$.
Since a map $\Psi:\bm \times [0,\infty)\to T^{\perp}\bm$ defined by $\Psi(x,s):=s u_{x}$ is smooth,
the set $\bar{A}$ is also a null set of $T^{\perp}\bm$.
By the definition of $\tau$,
the set $A$ is contained in $\inte M$.
From the smoothness of $\expp$,
it follows that
$A$ is a null set of $M$.

Since $\Phi^{-1}_{t}(\Omega)\setminus \bm=A\sqcup B$,
and since $A$ is a null set of $M$,
it suffices to show that
\begin{equation*}
\vol_{g}B=\int_{\Pi(\Omega)}\int_{I_{\Omega,t,x}}\,\theta(s,x)\,ds\,d\vol_{h}.
\end{equation*}
We put
\begin{equation*}
\bar{B}:=\bigcup_{x\in \Pi(\Omega)} \{s u_{x}\mid s\in I_{\Omega,t,x} \}.
\end{equation*} 
Note that $B=\expp(\bar{B})$.
The set $\bar{B}$ is contained in $TD_{\bm}\setminus 0(\tbp)$.
By Lemma \ref{lem:regular and injective},
the map $\expp|_{TD_{\bm}\setminus 0(T^{\perp}\bm)}$ is a diffeomorphism.
Hence,
by the coarea formula and the Fubini theorem,
\begin{equation*}
\vol_{g} \expp(\bar{B})=\int_{\Pi(\Omega)}\int_{I_{\Omega,t,x}}\,\theta(s,x)\,ds\,d\vol_{h}.
\end{equation*}
Since $B=\expp(\bar{B})$,
we arrive at the desired equation.
\end{proof}
Now,
we prove the following measure contraction inequality:
\begin{prop}\label{prop:mcp1}
Let $M$ be an $n$-dimensional,
connected complete Riemannian manifold with boundary with Riemannian metric $g$ such that
$\ric_{M}\geq (n-1)\kappa$ and $H_{\bm}\geq \lambda$.
For $t\in (0,1)$,
let $\Phi_{t}$ be the $t$-extension map from $\bm$.
Suppose that
a subset $\Omega$ of $M$ is measurable.
Then we have
\begin{equation*}\label{eq:mcp1}
\vol_{g} \Phi^{-1}_{t}(\Omega) \geq t\,\int_{\Omega}\,\frac{s^{n-1}_{\kappa,\lambda} \circ t \rho_{\bm}}{s^{n-1}_{\kappa,\lambda} \circ \rho_{\bm}}   d\vol_{g}.
\end{equation*}
\end{prop}
\begin{proof}
We may assume $\vol_{g} \Phi^{-1}_{t}(\Omega)<\infty$.
By Lemma \ref{lem:contraction and domain},
\begin{equation*}\label{eq:contraction integration formula2}
\vol_{g}\Phi^{-1}_{t}(\Omega)=\int_{\Pi(\Omega)}\int_{I_{\Omega,t,x}}\,\theta(s,x)\,ds\,d\vol_{h}.
\end{equation*}
From Lemma \ref{lem:comp2},
for all $x\in \Pi(\Omega)$ and $s\in I_{\Omega,t,x}$,
we derive
\begin{equation*}\label{eq:measure contraction comparison}
\frac{\theta(t^{-1}s,x)}{\theta(s,x)} \leq \frac{s^{n-1}_{\kappa,\lambda}(t^{-1}s)}{s^{n-1}_{\kappa,\lambda}(s)}.
\end{equation*}
It follows that
\begin{equation*}\label{eq:integration measure contraction comparison}
\vol_{g}\Phi^{-1}_{t}(\Omega)\geq \int_{\Pi(\Omega)}\int_{I_{\Omega,t,x}}\,\frac{s^{n-1}_{\kappa,\lambda}(s)}{s^{n-1}_{\kappa,\lambda}(t^{-1}s)}\theta(t^{-1}s,x) \,ds\,d\vol_{h}.
\end{equation*}
For $x \in \Pi(\Omega)$,
we put
\begin{equation*}
I_{\Omega,x}:=\{\, s \in (0,\tau(x))\mid \gamma_{x}(s)\in \Omega \,\}.
\end{equation*}
Note that
for each $x\in \Pi(\Omega)$,
the set $\{l \in (0,\tau(x)) \mid t l\in I_{\Omega,t,x}\}$ coincides with $I_{\Omega,x}$.
By putting $l:=t^{-1} s$ in the above inequality,
we have
\begin{equation*}
\vol_{g}\Phi^{-1}_{t}(\Omega)\geq t\int_{\Pi(\Omega)}\int_{I_{\Omega,x}}\,\frac{s^{n-1}_{\kappa,\lambda}(tl)}{s^{n-1}_{\kappa,\lambda}(l)}\theta(l,x) \,dl\,d\vol_{h}.
\end{equation*}

Now,
we put 
\begin{equation*}
\bar{U}:=\bigcup_{x\in \Pi(\Omega)} \{su_{x}\mid s\in I_{\Omega,x} \}.
\end{equation*}
We show $\expp(\bar{U})=\Omega \setminus \left(\cut \bm \cup \bm \right)$.
By the definition of $I_{\Omega,x}$,
we have $\expp(\bar{U})\subset \Omega \setminus \left(\cut \bm \cup \bm \right)$.
To show the opposite,
take $p\in \Omega\setminus \left(\cut \bm \cup \bm \right)$,
and take a foot point $x$ on $\bm$ of $p$.
From Lemma \ref{lem:foot vector},
it follows that $p=\expp(\rho_{\bm}(p)u_{x})$.
We see $x\in \Pi(\Omega)$.
Since $p$ does not belongs to $\cut \bm \cup \bm$,
we have $\rho_{\bm}(p) \in (0,\tau(x))$.
This implies $\rho_{\bm}(p)\in I_{\Omega,x}$.
Hence,
the set $\Omega \setminus \left(\cut \bm \cup \bm \right)$ is contained in $\expp(\bar{U})$.

The set $\bar{U}$ is contained in $TD_{\bm}\setminus 0(\tbp)$.
Lemma \ref{lem:regular and injective} implies that
the map $\expp|_{TD_{\bm}\setminus 0(T^{\perp}\bm)}$ is a diffeomorphism.
By the coarea formula and the Fubini theorem,
and by Lemma \ref{prop:vol of cut},
we have
\begin{align*}
t\int_{\Pi(\Omega)}\int_{I_{\Omega,x}}\,\frac{s^{n-1}_{\kappa,\lambda}(tl)}{s^{n-1}_{\kappa,\lambda}(l)}\theta(l,x) \,dl\,d\vol_{h}
&=t\,\int_{\expp(\bar{U})}\,\frac{s^{n-1}_{\kappa,\lambda} \circ t \rho_{\bm}}{s^{n-1}_{\kappa,\lambda} \circ \rho_{\bm}}   d\vol_{g}\\
&=t\,\int_{\Omega}\,\frac{s^{n-1}_{\kappa,\lambda} \circ t \rho_{\bm}}{s^{n-1}_{\kappa,\lambda} \circ \rho_{\bm}}   d\vol_{g}.
\end{align*}
Thus,
we arrive at the desired inequality.
\end{proof}
\subsection{Another proof of Theorem \ref{thm:volume comparison}}
For $r,R \in (0,\infty)$ with $r<R$,
we put $A_{r,R}(\bm):=B_{R}(\bm)\setminus B_{r}(\bm)$.

By using Proposition \ref{prop:mcp1},
we have the following:
\begin{lem}\label{lem:mcp2}
Let $M$ be an $n$-dimensional,
connected complete Riemannian manifold with boundary with Riemannian metric $g$ such that
$\ric_{M}\geq (n-1)\kappa$ and $H_{\bm}\geq \lambda$.
Let $t \in (0,1)$.
Suppose $\bm$ is compact.
Then for all $R\in (0,\bar{C}_{\kappa,\lambda}] \setminus \{\infty\}$ and $r\in (0,R)$,
we have
\begin{equation*}\label{eq:mcp2}
\frac{\vol_{g} A_{r,R}(\bm)}{\vol_{g} A_{t r,t R}(\bm)} \leq \left(t\, \inf_{s \in (r,R)}\frac{s^{n-1}_{\kappa,\lambda}(t s)}{s^{n-1}_{\kappa,\lambda}(s)} \right)^{-1}.
\end{equation*}
\end{lem}
\begin{proof}
Take $R\in (0,\bar{C}_{\kappa,\lambda}] \setminus \{\infty\}$ and $r\in (0,R)$.
Put $\Omega:=A_{r,R}(\bm)$.
Let $\Phi_{t}$ be the $t$-extension map from $\bm$.
For all $p \in \Phi^{-1}_{t}(\Omega)$,
we have
\begin{equation*}
\rho_{\bm}(p)=t\,\rho_{\bm}(\Phi_{t}(p))\in (t r,t R].
\end{equation*}
Hence,
$\Phi^{-1}_{t}(\Omega)$ is contained in $A_{t r,t R}(\bm)$.
Applying Proposition \ref{prop:mcp1} to $\Omega$,
we obtain
\begin{equation*}
\vol_{g} A_{t r,t R}(\bm) \geq \vol_{g} \Phi^{-1}_{t}(\Omega) \geq t\, \inf_{s \in (r,R)}\frac{s^{n-1}_{\kappa,\lambda}(t s)}{s^{n-1}_{\kappa,\lambda}(s)}\, \vol_{g} \Omega.
\end{equation*}
This proves the lemma.
\end{proof}
From Lemma \ref{lem:mcp2},
we derive the following:
\begin{lem}\label{lem:mcp3}
Let $M$ be an $n$-dimensional,
connected complete Riemannian manifold with boundary with Riemannian metric $g$ such that
$\ric_{M}\geq (n-1)\kappa$ and $H_{\bm}\geq \lambda$.
Suppose $\bm$ is compact.
Let $r_{2}\in (0,\bar{C}_{\kappa,\lambda}] \setminus \{\infty\}$,
and let $r_{1}\in (0,r_{2})$.
Put $t:=r_{1}/r_{2}$.
For $k\in\mathbb{N}$,
put $r:=t^{k}r_{2}$.
Then we have
\begin{equation*}\label{eq:mcp3}
\frac{\vol_{g} A_{r_{1},r_{2}}(\bm)}{\vol_{g} B_{r}(\bm)}\leq \left(\sum^{\infty}_{i=k} t^{i} \inf_{s \in (r_{1},r_{2})} \frac{s^{n-1}_{\kappa,\lambda}(t^{i}s)}{s^{n-1}_{\kappa,\lambda}(s)}\right)^{-1}.
\end{equation*}
\end{lem}
\begin{proof}
We see $B_{r}(\bm)\setminus \bm=\bigcup^{\infty}_{i=k}A_{t^{i}r_{1},t^{i}r_{2}}(\bm).$
Lemma \ref{lem:mcp2} implies
\begin{align*}
        \vol_{g} B_{r}(\bm)
&=   \sum^{\infty}_{i=k}\vol_{g} A_{t^{i}r_{1},t^{i}r_{2}}(\bm)\\
&\geq \vol_{g}A_{r_{1},r_{2}}(\bm) \left(\sum^{\infty}_{i=k} t^{i} \inf_{s \in (r_{1},r_{2})} \frac{s^{n-1}_{\kappa,\lambda}(t^{i}s)}{s^{n-1}_{\kappa,\lambda}(s)}\right).
\end{align*}
This completes the proof.
\end{proof}
By Lemma \ref{lem:mcp3},
we have the following volume estimate:
\begin{lem}\label{lem:mcp4}
Let $M$ be an $n$-dimensional,
connected complete Riemannian manifold with boundary with Riemannian metric $g$ such that
$\ric_{M}\geq (n-1)\kappa$ and $H_{\bm}\geq \lambda$.
Suppose $\bm$ is compact.
Let $t \in (0,1)$.
Take $l,m\in \mathbb{N}$ with $l<m$.
Then for all $r \in (0,\infty)$ with $t^{l-1} r\in (0,\bar{C}_{\kappa,\lambda}] \setminus \{\infty\}$,
we have
\begin{equation*}
       \frac{\vol_{g}B_{t^{l-1}r}(\bm)}{\vol_{g}B_{t^{m-1}r}(\bm)}
\leq \frac{\sum^{\infty}_{j=l}\sup_{s \in (t^{j}r,t^{j-1}r)}  s^{n-1}_{\kappa,\lambda}(s) (t^{j-1}r-t^{j}r)}{\sum^{\infty}_{i=m}\inf_{s \in (t^{i}r,t^{i-1}r)} s^{n-1}_{\kappa,\lambda}(s)(t^{i-1}r-t^{i}r)}.
\end{equation*}
\end{lem}
\begin{proof}
Fix $j \in \{l,\dots,m-1\}$.
By Lemma \ref{lem:mcp3},
we have
\begin{align*}
       \frac{\vol_{g}A_{t^{j}r,t^{j-1}r}(\bm)}{\vol_{g}B_{t^{m-1}r}(\bm)}
&\leq \left( \sum^{\infty}_{i=m-j}t^{i} \inf_{s\in (t^{j}r,t^{j-1}r)} \frac{s^{n-1}_{\kappa,\lambda}(t^{i}s)}{s^{n-1}_{\kappa,\lambda}(s)} \right)^{-1}\\
&\leq \left( \sum^{\infty}_{i=m-j}t^{i} \frac{\inf_{s \in (t^{j}r,t^{j-1}r)} s^{n-1}_{\kappa,\lambda}(t^{i}s)}{\sup_{s \in (t^{j}r,t^{j-1}r)}  s^{n-1}_{\kappa,\lambda}(s)} \right)^{-1}.
\end{align*}
Note that
we have
\begin{equation*}
       \left( \sum^{\infty}_{i=m-j}t^{i} \frac{\inf_{s \in (t^{j}r,t^{j-1}r)} s^{n-1}_{\kappa,\lambda}(t^{i}s)}{\sup_{s \in (t^{j}r,t^{j-1}r)}  s^{n-1}_{\kappa,\lambda}(s)} \right)^{-1}
  =   \frac{t^{j}\sup_{s \in (t^{j}r,t^{j-1}r)}  s^{n-1}_{\kappa,\lambda}(s)}{\sum^{\infty}_{i=m}t^{i}\inf_{s \in (t^{i}r,t^{i-1}r)} s^{n-1}_{\kappa,\lambda}(s)}.
\end{equation*}
It follows that
\begin{align*}
          \frac{\vol_{g}B_{t^{l-1}r}(\bm)}{\vol_{g}B_{t^{m-1}r}(\bm)}
&=     1+\sum^{m-1}_{j=l}\frac{\vol_{g}A_{t^{j}r,t^{j-1}r}(\bm)}{\vol_{g}B_{t^{m-1}r}(\bm)}\\
&\leq 1+\sum^{m-1}_{j=l} \frac{t^{j}\sup_{s \in (t^{j}r,t^{j-1}r)}  s^{n-1}_{\kappa,\lambda}(s)}{\sum^{\infty}_{i=m}t^{i}\inf_{s \in (t^{i}r,t^{i-1}r)} s^{n-1}_{\kappa,\lambda}(s)}\\
&\leq \frac{\sum^{\infty}_{j=l}t^{j}\sup_{s \in (t^{j}r,t^{j-1}r)}  s^{n-1}_{\kappa,\lambda}(s)}{\sum^{\infty}_{i=m}t^{i}\inf_{s \in (t^{i}r,t^{i-1}r)} s^{n-1}_{\kappa,\lambda}(s)}.
\end{align*}
This implies the lemma.
\end{proof}
Now,
we give another proof of Theorem \ref{thm:volume comparison}.
\begin{proof}[Proof of Theorem \ref{thm:volume comparison}]
Let $M$ be an $n$-dimensional,
connected complete Riemannian manifold with boundary with Riemannian metric $g$
such that $\ric_{M}\geq (n-1)\kappa$ and $H_{\bm} \geq \lambda$.
Suppose $\bm$ is compact.
Take $r,R\in (0,\infty)$ with $r\leq R$.
By Lemma \ref{lem:shinitai2},
we may assume $R\in (0,\bar{C}_{\kappa,\lambda}] \setminus \{\infty\}$ and $r<R$.
Put $r_{0}:=R r$.
Take a sufficiently large $N \in \mathbb{N}$ such that $N^{-1}\log r \in (0,1)$.
We put $t:=1-(\log r/N)$,
and
\begin{equation*}
l:=N+1,\quad m:=\min \left\{i\in \mathbb{N} \mid i \geq N (\log R/\log r)+1\right\}.
\end{equation*}
We have $l<m$ and $t^{m-1}r_{0}\leq r$.
Note that
if $N\to \infty$,
then $t^{l-1}r_{0}\to R$ and $t^{m-1}r_{0}\to r$.
From Lemma \ref{lem:mcp4},
it follows that
\begin{align*}
       \frac{\vol_{g}B_{t^{l-1}r_{0}}(\bm)}{\vol_{g}B_{r}(\bm)}
&\leq \frac{\vol_{g}B_{t^{l-1}r_{0}}(\bm)}{\vol_{g}B_{t^{m-1}r_{0}}(\bm)}\\
&\leq \frac{\sum^{\infty}_{j=l}\sup_{s \in (t^{j}r_{0},t^{j-1}r_{0})}  s^{n-1}_{\kappa,\lambda}(s) (t^{j-1}r_{0}-t^{j}r_{0})}{\sum^{\infty}_{i=m}\inf_{s \in (t^{i}r_{0},t^{i-1}r_{0})} s^{n-1}_{\kappa,\lambda}(s)(t^{i-1}r_{0}-t^{i}r_{0})}.
\end{align*}
Letting $N\to \infty$,
we have
\begin{equation*}
       \frac{\vol_{g}B_{R}(\bm)}{\vol_{g}B_{r}(\bm)}
\leq \frac{\int^{R}_{0}\,s^{n-1}_{\kappa,\lambda}(s)\,ds}{\int^{r}_{0}\,s^{n-1}_{\kappa,\lambda}(s)\,ds}.
\end{equation*}
Thus,
we obtain Theorem \ref{thm:volume comparison}.
\end{proof}


\begin{thebibliography}{99}
 \bibitem{B}        A.L. Besse: Einstein Manifolds, Springer-Verlag, New York, 1987.
 \bibitem{BC}     R. Bishop and R. Crittenden: Geometry of Manifolds, Academic Press, 1964.
 \bibitem{BBI}    D. Burago, Y. Burago, S. Ivanov: A Course in Metric Geometry, Graduate Studies in Math. 33, Amer. Math. Soc., 2001.
 \bibitem{C}       E. Calabi: \textit{An extention of E. Hopf's maximum principle with an application to Riemannian geometry}, Duke Math. J. 25 (1957), 45--56.
 \bibitem{Ch}     I. Chavel: Eigenvalues in Riemannian Geometry, Academic Press, 1984.
 \bibitem{Che}   J. Cheeger: \textit{A lower bound for the smallest eigenvalue of the Laplacian}, Problems in analysis, a symposium in honor of S. Bochner, Princeton University Press, Princeton, 1970, 195--199.
  \bibitem{CC}   J. Cheeger and T.H. Colding: \textit{Lower bounds on Ricci curvature and the almost rigidity of warped products}, Ann. of Math. (2) 144. 1 (1996), 189--237.
 \bibitem{CG}    J. Cheeger and D. Gromoll: \textit{The splitting theorem for manifolds of nonnegative Ricci curvature}, J. Differential Geom. 6 (1971), 119--128.
 \bibitem{CK}    C. Croke and B. Kleiner: \textit{A warped product splitting theorem}, Duke Math. J. 67 (1992), 571--574.
 \bibitem{dCX}   M.P. do Carmo and C. Xia: \textit{Rigidity theorems for manifolds with boundary and nonnegative curvature}, Result. Math. 40 (2001), 122--129.
 \bibitem{EH}   J. Eschenburg and E. Heintze: \textit{An elementary proof of the Cheeger-Gromoll splitting theorem}, Ann. Global Anal. Geom. 2.2 (1984), 141--151.
 \bibitem{F}       H. Federer: Geometric Measure Theory, Springer-Verlag, New York, 1969.
 \bibitem{FF}     H. Federer and W.H. Fleming: \textit{Normal and integral currents}, Ann. of Math. 72 (1960), 458--520.
 \bibitem{G}      J. Ge: \textit{Comparison theorems for manifold with mean convex boundary}, Comm. Contemp. Math (2014), online.
 \bibitem{GT}    D. Gilbarg and N.S. Trudinger: Elliptic Partial Differential Equations of Second Order, Springer-Verlag, 1983.
  \bibitem{Gr1}   M. Gromov: \textit{Curvature, diameter and Betti numbers}, Comment. Math. Helv. 56 (1981), 179--195.
 \bibitem{Gr2}   \bysame: Structures metriques pour les varieties Riemanniennes, Cedic-Fernand Nathan, Paris, 1981.
 \bibitem{HK}    E. Heintze and H. Karcher: \textit{A general comparison theorem with applications to volume estimates for submanifolds}, Ann. Sci. Ecole Norm. Sup. 11 (1978), 451--470.
 \bibitem{H}      S. Helgason: Differential Geometry, Lie Groups, and Symmetric Spaces,  Academic press, 1978.
 \bibitem{I}        R. Ichida: \textit{Riemannian manifolds with compact boundary}, Yokohama Math. J. 29. 2 (1981), 169--177.
 \bibitem{K1}   A. Kasue: \textit{On Laplacian and Hessian comparison theorems}, Proc. Japan Acad. 58 (1982), 25--28.
 \bibitem{K2}   \bysame: \textit{A Laplacian comparison theorem and function theoretic properties of a complete Riemannian manifold}, Japanese J. Math. New Series 8. 2 (1982), 309--341.
 \bibitem{K3}   \bysame: \textit{Ricci curvature, geodesics and some geometric properties of Riemannian manifolds with boundary}, J. Math. Soc. Japan 35. 1 (1983), 117--131.
 \bibitem{K4}   \bysame: \textit{On a lower bound for the first eigenvalue of the Laplace operator on a Riemannian manifold}, Ann. Sci. Ecole Norm. Sup. 17. 1 (1984), 31--44.
 \bibitem{K5}   \bysame: \textit{Applications of Laplacian and Hessian Comparison Theorems}, Advanced Studies in Pure Math. 3 (1984), 333--386.
 \bibitem{KN}   S. Kawai and N. Nakauchi: \textit{The first eigenvalue of the $p$-Laplacian on a compact Riemannian manifold}, Nonlinear Anal. 55. 1--2 (2003), 33--46.
 \bibitem{LW}   H. Li and Y. Wei: \textit{Rigidity theorems for diameter estimates of compact manifold with boundary}, International Mathematics Research Notices (2014), rnu052, 18pages.
 \bibitem{L}      M. Li: \textit{A sharp comparison theorem for compact manifolds with mean convex boundary}, J. Geom. Anal. 24 (2014), 1490--1496.
 \bibitem{Li}     P. Li: Geometric Analysis, Cambridge University Press, 2012.
 \bibitem{LY}   P. Li and S.T. Yau: \textit{Estimates of eigenvalues of a compact Riemannian manifold}, Proc. Symp. Pure Math. 36 (1980), 205--239.
 \bibitem{O1}   S. Ohta: \textit{On the measure contraction property of metric measure spaces}, Comm. Math. Helv. 82 (2007), 805--828.
 \bibitem{O2}   \bysame: \textit{Products, cones, and suspensions of spaces with the measure contraction property}, J. Lond. Math. Soc. 76 (2007), 225--236.
 \bibitem{P}     R. Perales: \textit{Volumes and limits of manifolds with Ricci curvature and mean curvature bound}, arXiv preprint arXiv:1404.0560v3 (2014).
 \bibitem{S}     T. Sakai: Riemannian Geometry, Translations of Mathematical Monographs 149, Amer. Math. Soc, 1996.
 \bibitem{St}     K.-T. Sturm: \textit{Diffusion processes and heat kernels on metric spaces}, Ann. Prob. 26 (1998), 1--55.
 \bibitem{X}     C. Xia: \textit{Rigidity of compact manifolds with boundary and nonnegative Ricci curvature}, Proc. Amer. Math. Soc. 125 (1997), 1801--1806.
 \bibitem{Z1}     H. Zhang: \textit{Lower bounds for the first eigenvalue of the $p$-Laplace operator on compact manifolds with positive Ricci curvature}, Nonlinear. Anal. 67. 3 (2007), 795--802.
 \bibitem{Z2}    \bysame: \textit{Lower bounds for the first eigenvalue of the $p$-Laplace operator on compact manifolds with nonnegative Ricci curvature}, Adv. Geom. 7. 1 (2007), 145--155.
\end{thebibliography}
\end{document}